\documentclass[11pt,reqno]{amsart}
\usepackage[leqno]{amsmath}
\usepackage{amscd,amssymb,palatino}
\usepackage[english]{babel}
\usepackage{caption}
\usepackage{etex}
\usepackage[leqno]{amsmath}
\usepackage{amssymb}
\usepackage{color}
\usepackage{shadow}
\usepackage{epsfig}
\usepackage{epic}
\usepackage[dvipsnames]{xcolor}
\usepackage{graphics}
\usepackage{graphicx}
\usepackage{psfrag}
\usepackage{url}
\usepackage{hyperref}
\usepackage{bm}
\usepackage{comment}
\usepackage{parskip}
\usepackage{float}

\usepackage {tikz}
\usetikzlibrary {automata, positioning}
\definecolor {processblue}{cmyk}{0.96,0,0,0}

\usepackage{cleveref} 
\usepackage{tensor}
\usepackage{enumitem}

\usepackage{amsthm, amscd, amsfonts, amssymb} 
\usepackage{mathtools} 
\usepackage{color}
\usepackage[all]{xy}

\theoremstyle{plain}
\newtheorem{proposition}{Proposition}[section]
\newtheorem{theorem}[proposition]{Theorem}
\newtheorem{corollary}[proposition]{Corollary}

\newtheorem{conjecture}[proposition]{Conjecture}

\theoremstyle{definition}
\newtheorem{definition}[proposition]{Definition}

\theoremstyle{remark}
\newtheorem{remark}[proposition]{Remark}

\newcommand\restr[2]{{
  \left.\kern-\nulldelimiterspace
  #1
  \vphantom{\big|}
  \right|_{#2}
  }}

\DeclarePairedDelimiterX{\innerp}[2]{\langle}{\rangle}{#1,#2}

\newcommand{\NN}{\mathbb{N}}
\newcommand{\ZZ}{\mathbb{Z}}

\newcommand{\TT}{\mathbb{T}}

\newcommand{\RR}{\mathbb{R}}

\newcommand{\DD}{\mathbb{D}}
\renewcommand{\SS}{\mathbb{S}}

\newcommand{\fX}{\mathfrak{X}}

\renewcommand{\div}{\textup{div}}
\newcommand{\curl}{\textup{curl}}
\newcommand{\Xdiv}{\fX^{\textup{div}}}

\newcommand{\id}{\textup{id}}

\usepackage{setspace}
\usepackage{graphicx}

\usepackage{tikz-cd}
\usepackage{circuitikz}
\usetikzlibrary{shapes.geometric, arrows}

\tikzstyle{mytheorembox} = [draw=vdgreen, fill=blue!20, very thick, rectangle, rounded corners, inner sep=10pt, inner ysep=15pt]
\tikzstyle{mytheoremfancytitle} =[fill=vdgreen, text=white]

\definecolor{vdblue}{rgb}{0,0,.3}
\definecolor{dblue}{rgb}{0,0,.7}
\definecolor{lblue}{rgb}{.3,.3,1}
\definecolor{vlblue}{rgb}{.7,.7,1}
\definecolor{vvlblue}{rgb}{.9,.9,1}

\definecolor{vdred}{rgb}{.3,0,0}
\definecolor{dred}{rgb}{.7,0,0}
\definecolor{lred}{rgb}{1,.3,.3}
\definecolor{vlred}{rgb}{1,.7,.7}

\definecolor{vdgreen}{rgb}{0,.2,0}
\definecolor{dgreen}{rgb}{0,.4,0}
\definecolor{lgreen}{rgb}{.3,1,.3}
\definecolor{vlgreen}{rgb}{.7,1,.7}

\definecolor{lyellow}{rgb}{1,1,.3}

\definecolor{gray1}{rgb}{0.22,0.22,0.22}
\definecolor{gray2}{rgb}{0.28,0.28,0.28}
\definecolor{gray3}{rgb}{0.36,0.36,0.36}
\definecolor{gray4}{rgb}{0.44,0.44,0.44}
\definecolor{gray5}{rgb}{0.52,0.52,0.52}
\definecolor{gray6}{rgb}{0.6,0.6,0.6}
\definecolor{gray7}{rgb}{0.68,0.68,0.68}
\definecolor{gray8}{rgb}{0.76,0.76,0.76}

\definecolor{color1}{rgb}{1,0,0}
\definecolor{color2}{rgb}{0.98,0,0.816}
\definecolor{color3}{rgb}{0.717,0,1}
\definecolor{color4}{rgb}{0,0,1}
\definecolor{color5}{rgb}{0,1,1}
\definecolor{color6}{rgb}{0,1,0}
\definecolor{color8}{rgb}{1,1,0}
\definecolor{color7}{rgb}{1,0.651,0}

\begin{document}
\title{Universality in computable dynamical systems: Old and new}
\author{Ángel González-Prieto}
\address{ \'Angel Gonz\'alez-Prieto, Departamento de \'Algebra, Geometr\'ia y Topolog\'ia, Facultad de CC.\ Matem\'aticas, Universidad Complutense de Madrid, 28040 Madrid, Spain and Instituto de Ciencias Matem\'aticas (CSIC-UAM-UCM-UC3M),
28049 Madrid, Spain.}
 \thanks{Eva Miranda is the corresponding author for this article. Send any correspondence to: eva.miranda@upc.edu.}
 \thanks{\'Angel Gonz\'alez-Prieto is supported by the Ministerio de Ciencia e Innovaci\'on Grant PID2021-124440NB-I00 (Spain)}
\email{angelgonzalezprieto@ucm.es}
\author{Eva Miranda }\address{Eva Miranda,
Laboratory of Geometry and Dynamical Systems-SYMCREA, Department of Mathematics, EPSEB, Universitat Polit\`{e}cnica de Catalunya-IMTech
in Barcelona and
\\ CRM Centre de Recerca Matem\`{a}tica, Edifici C, Campus de Bellaterra, 08193 Bellaterra, Barcelona.
 }\thanks{Eva Miranda is funded by the Catalan Institution for Research and Advanced Studies via an ICREA Academia Prize 2021 and by the Alexander Von Humboldt Foundation via a Friedrich Wilhelm Bessel Research Award and by the AGAUR project 2021 SGR 00603. She is also supported by the Spanish State
Research Agency, through the Severo Ochoa and Mar\'{\i}a de Maeztu Program for Centers and Units
of Excellence in R\&D (project CEX2020-001084-M),  and under grant reference PID2023-146936NB-I00 funded by MICIU/AEI/
10.13039/501100011033 and by ERDF/EU} 
 \email{eva.miranda@upc.edu}
 \author{Daniel Peralta-Salas} \address{Daniel Peralta-Salas, Instituto de Ciencias Matem\'aticas, Consejo Superior de Investigaciones Cient\'ificas,
28049 Madrid, Spain.}
\thanks{Daniel Peralta-Salas is supported by the grants CEX2023-001347-S, RED2022-134301-T and PID2022-136795NB-I00 funded by MCIN/AEI/ 10.13039/501100011033.}
\thanks{All authors are funded by the project “Computational, dynamical and geometrical complexity in fluid dynamics”, Ayudas Fundación BBVA a Proyectos de Investigación Científica 2021, and by Bilateral AEI-DFG project: Celestial Mechanics, Hydrodynamics, and Turing Machines  with reference codes PCI2024-155042-2 and PCI2024-155062-2}
\email{dperalta@icmat.es}

\begin{abstract} 
The relationship between computational models and dynamics has captivated mathematicians and computer scientists since the earliest conceptualizations of computation. Recently, this connection has gained renewed attention, fueled by T. Tao's programme aiming to discover blowing-up solutions of the Navier-Stokes equations using an embedded computational model. In this survey paper, we review some of the recent works that introduce novel and exciting perspectives on the representation of computability through dynamical systems. Starting from dynamical universality in a classical sense, we shall explore the modern notions of Turing universality in fluid dynamics and Topological Kleene Field Theories as a systematic way of representing computable functions by means of dynamical bordisms. Finally, we will discuss some important open problems in the area.
\end{abstract}

\maketitle
\bigskip

\section{Introduction}

The relationship between dynamical systems and computation has generated a rich interplay between mathematics and theoretical computer science since its very inception. In recent years, ideas of universality have emerged that allow one to embed arbitrary computations into the flow of a dynamical system. This correspondence not only deepens our understanding of physical processes, like fluid flows or Hamiltonian dynamics, but also challenges our notions of what it means to compute.

In this survey paper, we review some of the recent developments of the area. Starting from universality in a classical sense, we show how very general dynamics can be embedded in restricted systems by increasing the dimension of the ambient space, such as in Hamiltonian systems or potential wells. 

Even more interestingly, in this work we will pay special attention to a particularly remarkable non-linear dynamical system, namely the motion of an ideal fluid on a Riemannian manifold $(M,g)$, as described by the so-called Euler equations 
\begin{equation}\label{eq:intro}
    \dot{X} + \nabla_X X = -\nabla p, \qquad \div(X) = 0,
\end{equation}
where $X$ is the unknown velocity field of the fluid and $p$ is the pressure function exerted by the fluid. Solutions to the Euler equations (\ref{eq:intro}) are known to be extremely complex and include some of the most intriguing phenomena such as chaos or highly knotted orbits.

One of the most important open problems in dynamical systems and partial differential equations is determining whether the Euler equations (and their viscous version, the Navier-Stokes equations) admit solutions that blow-up in finite time. To address this problem, in 2017, T.\ Tao envisaged a strategy based on universal dynamics. In his seminal paper \cite{tao2017universality}, he showed that any conservative quadratic dynamics in the Euclidean space can be embedded as the Euler flow inside an invariant finite-dimensional subspace of the space of volume-preserving vector fields on $M$. This was the first evidence that Euler flows must be very flexible, and very involved dynamics should be embeddable as solutions. As a consequence of this result, Tao conjectured that it should be possible to embed the dynamics of a computer inside the Euler equations, in such a way that any program could be reproduced by means of the flow. In this way, he speculated that a blowing-up solution could be obtained by using this computational universality, a dream that has not yet been achieved.

In the last years, Tao's programme has fostered research on the interplay between computability and dynamics. In the following sections, we review the contribution by the last two authors of this work jointly with R.\ Cardona and F.\ Presas, which shows that the computational procedure of any Turing machine can be represented as the Poincar\'e first return map of a stationary Euler flow on a $3$-dimensional manifold \cite{CMPP}. For this purpose, we shall review how a Turing machine can be seen as a discrete dynamical system, and how this discrete system can be represented as an area-preserving diffeomorphism of a disc, following the ideas of Moore \cite{moore1991generalized}. Using these insights on the dynamics of computation, the result follows by observing that this diffeomorphism generates a flow in its mapping torus of a very special type, known as a contact flow. In this way, using the correspondence unveiled by D. Sullivan and developed by J.\ Etnyre and R.\ Ghrist \cite{etnyre1997contact} that allows us to translate contact structures into particular (steady) Euler flows, this flow can be promoted to a genuine solution of the Euler equations.

In this direction, a new research line has recently been opened based on the dynamical representation of alternative models of computation. According to the celebrated Church--Turing thesis, every model of computability is equivalent to a Turing machine, but its dynamics and operation procedures might be very different. In this direction, we will outline some recent developments based on the dynamics of computable functions or, more precisely, partial recursive functions. In this sense, in the work \cite{GPMPS}, the authors show that any computable function can be represented by means of a flow on a bordism between two discs. Such bordisms may exhibit a very involved topology, such as loops and bifurcations, capturing the idea that the flow must revolve around the manifold in complicated ways to perform the whole calculation in one go. We refer to these constructions as Topological Kleene Field Theory (TKFT) after S.\ Kleene, who in \cite{kleene1936general} introduced the partial recursive functions, and in analogy with Topological Quantum Field Theories.

In this sense, we shall review in this paper the fundamental ingredients applied to the construction of the TKFT, highlighting how the dynamics execute very simple steps, whereas the topology of the bordism determines the globally difficult behaviour. Additionally, we shall outline some important applications of this TKFT to the complexity theory of computable functions. Indeed, these dynamical systems computing a partial recursive function enable a new way of looking at complexity from a purely topological point of view, by means of algebraic invariants attached to the representing bordism. The relation between this topological complexity and classical complexity will be the subject of future research.


The paper will finish with some natural open problems. Among them, it is particularly interesting the relation of these dynamical systems with quantum supremacy, namely the idea that quantum devices can outperform classical computers on specific tasks. Specifically, the results discussed in this paper prompt the natural question: Can continuous dynamical systems be engineered to surpass the efficiency of classical Turing machines? Could a new class of continuous models, perhaps inspired by fluid dynamics or topological field theories, provide a route to more efficient computation? With these techniques, would it be possible to beat quantum supremacy with dynamics-based computational devices?






\section{Universality in Dynamical Systems}\label{S:uds}

A key concept in understanding the complexity of a family of dynamical systems is the notion of \emph{universality}. In a broad sense, a family $\mathcal{U}$ of dynamical systems will be said to be universal for a class $\mathcal{C}$ of dynamical systems if every element of $\mathcal{C}$ can ``be embedded into'' a dynamical system of $\mathcal{U}$. In the case of smooth manifolds and vector fields, this can be formulated very explicitly following Tao~\cite{Tauuniversality}.

{Throughout this work, by a smooth manifold we mean a Hausdorff, second countable topological space that is locally Euclidean, and such that the change of coordinates are smooth. If $\varphi: M \hookrightarrow N$ is a differentiable map, its differential at a point $p \in M$ is the linear map between tangent spaces ${\varphi_*}_p: T_pM \to T_{\varphi(p)}N$. In particular, if $X$ is a vector field on $M$, then $\varphi_*(X)$ is the vector field on the image of $\varphi$ defined by $\varphi_*(X)_{\varphi(p)} = {\varphi_{*}}_p(X_p)$.
}

\begin{definition}[\cite{Tauuniversality}]\label{def:universality}
Let $\mathcal{C} = \{(M, X)\}$ be a collection of smooth manifolds $M$ with a vector field $X$ on it. We will say that another class $\mathcal{U} = \{(N, Y)\}$ is \emph{universal} for $\mathcal{C}$ if for every $(M, X) \in \mathcal{C}$, there exists $(N, Y) \in \mathcal{U}$ and an embedding $\varphi: M \hookrightarrow N$ such that $\varphi_*(X) = Y|_{\varphi(M)}$.
\end{definition}

\begin{remark}
The condition $\varphi_*(X) = Y|_{\varphi(M)}$ implies, in particular, that the flowlines of $X$ on $M$ are mapped into flowlines of $Y$ in $N$. In this sense, the dynamical system $(N, Y)$ has $\varphi(M)$ as an invariant submanifold and the flow there coincides with that of $X$. 
\end{remark}

Universality implies that, in a sense, the class $\mathcal{U}$ is as complex as $\mathcal{C}$, if not more. In other words, if we want to classify the dynamical systems in $\mathcal{C}$, we in particular need to classify the dynamics in some classes of invariant submanifolds of the universal class $\mathcal{U}$.

The most basic example of universality is based on classical mechanics. Recall that a symplectic manifold is a manifold $N$ endowed with a non-degenerate closed $2$-form $\omega \in \Omega^2(N)$. In this context, a smooth function $H: N \to \RR$ defines a unique vector field $X_H$ on $N$ determined by the condition $\iota_{X_H}\omega = dH$, called the \emph{Hamiltonian vector field} associated to $H$. Here $\iota_{X_H} \omega$ denotes the contraction $\iota_{X_H} \omega(Y_1, \ldots, Y_{n-1}) = \omega(X_H, Y_1, \ldots, Y_{n-1})$, and {$d: \Omega^{\bullet}(N) \to \Omega^{\bullet+1}(N)$ is the exterior differential}. These Hamiltonian systems model the dynamics of classical mechanical systems according to Hamilton's equations.

\begin{proposition}\label{prop:hamiltonian-universal}
	The class of Hamiltonian systems is universal for all smooth vector fields.
\end{proposition}

\begin{proof}
    Let $M$ be a smooth manifold with a smooth vector field $X$ on it. Consider $N = T^*M$ the cotangent bundle of $M$ and define the function $H: T^*M \to \mathbb{R}$ given by $H(q, p) = p(X_q)$, where $q \in M$ and $p \in T^*_qM$.  We embed $M$ as the zero section $\varphi: M \hookrightarrow T^*M$, given by $\varphi(q) = (q, 0)$. 

    Recall that $T^*M$ is a symplectic manifold with the $2$-form $\omega$ given in local coordinates $(q^i, p^i)$ by $\omega = \sum_i dq^i \wedge dp^i$. Therefore, if $X = \sum_k x^k\partial_{q_k}$ on $M$ we have $\varphi_*(X) = \sum_k x^k \partial_{q^k}$ on $T^*M$ and thus
    $$
        \iota_{\varphi_*(X)} \omega = \sum_i  dq^i(\varphi_*(X))dp^i - dp^i(\varphi_*(X))dq^i = \sum_i x_i dp^i. 
    $$
    Analogously, since in coordinates we have $H = \sum_i p^i x^i$, we get
    $$
        dH = \sum_i \sum_k \left[\frac{\partial}{\partial p^k}\left(p^i x^i\right)dp^k + \frac{\partial}{\partial q^k}\left(p^i x^i\right)dq^k\right] = \sum_i x^i dp^i + \sum_{i,k} p^i\frac{\partial x_i}{\partial q^k}dq^k.
    $$
    Thus, $dH|_{\varphi(M)} = \iota_{\varphi_*(X)} \omega$, showing that $\varphi_*(X)$ coincides with the Hamiltonian vector field associated to $H$ on $\varphi(M)$, as we wanted to prove.
\end{proof}

\begin{remark}
    As a direct corollary of the proof of Proposition \ref{prop:hamiltonian-universal} we observe that we can even restrict our universal class to Hamiltonian systems on phase spaces $T^*M$.
\end{remark}

Observe that this universality implies, in particular, that the complexity of Hamiltonian systems on submanifolds of half the dimension of the ambient symplectic manifold is at least as rich as that of the whole class of smooth dynamical systems.

\subsection{Potential dynamics}

Despite the importance of Proposition \ref{prop:hamiltonian-universal}, the Hamiltonian system considered in the proof may not be physically realizable in terms of potential theory.

For this reason, we focus on Hamiltonian functions arising from a potential, as these are more physically relevant. Hence, we shall consider a compact Riemannian manifold $(M, g)$ with a smooth function $V: M \to \RR$ on it, called the \emph{potential}. This defines the Hamiltonian function on $T^*M$ given by
$$
    H_V(q,p) = \frac{1}{2}|p|^2 + V(q).
$$
Here, the kinetic energy term $\frac{1}{2}|p|^2 = \frac{1}{2}g(p, p)$ is computed using the induced Riemannian metric on the cotangent bundle. The Hamiltonian $H_V$ represents the energy of a particle under the effect of the potential $V$, and thus the flow of the corresponding Hamiltonian vector field describes the trajectory of the particle. Indeed, if $q: I \subseteq \RR \to M$ is such a flow, Hamilton equations reduce in this case to Newton's equation
\begin{equation}\label{eq:potential-flow}
    \ddot{q} = -\nabla V(q).
\end{equation}

Therefore, we can consider the more restricted subclass $\mathcal{C}_{\textrm{p}} = \{(T^*M, X_{H_V})\}$ of potential-type Hamiltonian systems for some Riemannian manifold $(M, g)$ and potential $V$.

\begin{theorem}[{\cite[Theorem 1.6]{Tauuniversality}}]\label{thm:universality-pot}
	The class \(\mathcal{C}_{\textup{p}}\) is universal for smooth dynamical systems \((M, X)\) on a compact manifold for which there exists a 1-form \(\lambda \in \Omega^1(M)\) such that \(\lambda(X)_p > 0\) for all $p \in M$, and the Lie derivative \(L_X\lambda\) is exact.
\end{theorem}

\begin{remark}
    Since $\lambda(X) > 0$, it follows that the vector field $X$ is, in particular, nowhere vanishing.
\end{remark}

Observe that the requirement of the existence of such a $1$-form $\lambda \in \Omega^1(M)$ is natural, since potential-type Hamiltonian systems satisfy this property. Indeed, let us consider the Liouville form $\lambda = \sum_i p_i dq_i$ on $T^*M$, for which the symplectic form is $\omega = -d\lambda$. 
If $X$ is the Hamiltonian vector field associated to a potential-type Hamiltonian $H_V$, we have
$$
    L_X\lambda = \iota_Xd\lambda + d(\lambda(X)) = -\iota_X\omega + d(\lambda(X)) = d(\lambda(X) - H),
$$
where the last equality follows from the definition of Hamiltonian vector field, so $L_X\lambda$ is exact.

For the positivity condition, recall that in orthogonal local coordinates, we have
$$
    X = \sum_i p^i \partial_{q^i} - \sum_i\frac{\partial V}{\partial q^i} \partial_{p^i},
$$
and thus $\lambda(X) = |p|^2 \geq 0$.
This estimate can be improved if we restrict our attention to a submanifold $N$ contained in a regular level set of the Hamiltonian function (where $X$ is nowhere vanishing). Since the flow $(q(t), p(t))$ of $X$ satisfies $\dot{q}(t) = p(t)$, if $\lambda(X_{q(t), p(t)}) = 0$ for $t$ in an open set, then $q(t)$ is constant on that open set and thus $X$ itself identically vanishes there, contradicting the fact that $X$ does not vanish in $N$. This implies that the points $x \in N$ with $\lambda(X_x) = 0$ must be isolated in their orbit.
Therefore, we can arrange $\lambda(X) > 0$ by modifying the $1$-form by averaging around a small interval, as detailed in~\cite{Tauuniversality}.

The proof of the universality of this class of systems is presented in \cite[Section 2]{Tauuniversality}. It is based on a suitable choice of a Riemannian metric on $M$ that allows us to write the flow as a potential form.

\subsection{$\infty$-potential Dynamics}

As a variation of the flow (\ref{eq:potential-flow}), we can consider the ``$\infty$-potential'' equation
\begin{equation}\label{eq:infinity-potential}
    \partial_{tt}q - \Delta_{\mathbb{T}^d} q = - \nabla V(q),
\end{equation}
for a smooth function \(q: \RR \times \mathbb{T}^d \to M\), where $\mathbb{T}^d$ is the $d$-dimensional torus. 

It is worth noticing that solutions to the $\infty$-potential equation can be understood as Hamiltonian flows on an infinite dimensional manifold. To be precise, let us consider the space
\[
\mathcal{M}_M = C^\infty(\mathbb{T}^d, T^*M)
\]
{of smooth functions from $\mathbb{T}^d$ to the cotangent bundle $T^*M$}. On this space, we define the functional $\mathcal{H}_V: \mathcal{M}_M \to \RR$ that, for $f: \mathbb{T}^d \to T^*M$ given as $f(x) = (q(x), p(x))$, it computes
\[
\mathcal{H}_V(f) = \int_{\mathbb{T}^d} \Big(\frac{1}{2}|p(x)|^2 + \frac{1}{2}\sum_{k=1}^d \left|\partial_{x_k}q(x)\right|^2 + V(q(x))\Big)dx.
\]
Then, the flow of the Hamiltonian vector field associated to $\mathcal{H}_V$ is given by a function $u: I \subseteq \RR \to C^\infty(\mathbb{T}^d, T^*M)$ or, equivalently, a pair $(q(t, x), p(t, x)) \in T^*M$ for each $t \in I$ and $x \in \mathbb{T}^d$ satisfying the equations in orthogonal coordinates
$$
    \dot{q} = p, \qquad \dot{p} = \Delta_{\mathbb{T}^d} q - \nabla V(q),
$$
which are equivalent to (\ref{eq:infinity-potential}).

This defines a new class $\mathcal{C}_{\infty\textup{-p}} = \{(\mathcal{M}_M, X_{\mathcal{H}_V})\}$ of $\infty$-potential dynamical systems. In this direction, the main result in \cite{Tauuniversality} is the following analogue to Theorem \ref{thm:universality-pot} in the infinite-dimensional setting.

\begin{theorem}[{\cite[Theorem 1.6]{Tauuniversality}}]\label{thm:universality-infinite-pot}
	The class \(\mathcal{C}_{\infty\textup{-p}}\) is universal for smooth dynamical systems \((M, X)\) on a compact manifold for which there exists a 1-form \(\lambda \in \Omega^1(M)\) such that \(\lambda(X)_p > 0\) for all $p \in M$ and the Lie derivative \(L_X\lambda\) is exact.
\end{theorem}

However, in the context of Theorem \ref{thm:universality-infinite-pot}, the universality is subtler since the representation of a dynamical system $(M, X)$ as an $\infty$-potential system corresponds to an embedding (in the Gateaux sense) into the infinite-dimensional manifold
$$
    \varphi: M \hookrightarrow \mathcal{M}_N
$$
for some Riemannian manifold $N$. Analogously, the finite-dimensional vector field $X$ becomes under the embedding the infinite-dimensional Hamiltonian vector field $X_{\mathcal{H}_V}$.

The key takeaway is that universality must be understood in a broader sense than Definition \ref{def:universality}, allowing embeddings into potentially infinite-dimensional spaces.

\section{Universality in Euler Flows}

In this section, we explore how universality manifests in more complex systems, such as solutions of inviscid fluid equations on manifolds.

\begin{definition}
Let \((M, g)\) be a Riemannian manifold. An \emph{Euler flow} is a pair $(X, p)$, where $X$ is a time-dependent, volume-preserving vector field and $p: M \to \RR$ is a smooth function satisfying the equation
\begin{equation}\label{eq:Euler}
\dot{X} + \nabla_X X = -\nabla p.
\end{equation}
Equation (\ref{eq:Euler}) is known as the \emph{Euler equation}.
\end{definition}

{ Here, $\nabla_XY$ denotes the covariant derivative of a vector field $Y$ with respect to $X$ by means of the Levi-Civita connection associated to $g$. Additionally, $\nabla p$ is the gradient vector field of the function $p$, defined by satisfying $dp = g(\nabla p, -)$.}

\begin{remark}
    The property of the vector field $X$ being volume-preserving can be rephrased using the \emph{divergence} operator. For this purpose, let $\alpha = g(X, -)$ be the dual $1$-form. The divergence of $X$ is the function
    $$
        \div(X) = \star d\star\alpha,
    $$
    where $\star: \Omega^\bullet(M) \to \Omega^{\dim M - \bullet}(M)$ is the Hodge star operator. This operator satisfies the property that the Lie derivative of the Riemannian volume form $\mu$ satisfies $L_X \mu = \textup{div}(X)\mu$. In this way, $X$ is volume-preserving if and only if $\div(X) = 0$. For this reason, these vector fields are also known as \emph{divergence-free}. Using Cartan's formula, this is also equivalent to $d\iota_X\mu = 0$. Notice that Equation (\ref{eq:Euler}) does not imply that $X$ is volume-preserving, so an Euler flow must satisfy the two conditions simultaneously.
\end{remark}

\begin{remark}
    Since the function $p$ is determined by the vector field $X$ up to a constant, it is customary to abuse notation and call $X$ itself the Euler flow.
\end{remark}

\begin{remark}
The vector field $X$ of an Euler flow models the velocity field of an incompressible inviscid fluid on $M$ subject to a pressure $p$. In this sense, Euler equations are a simplification of the so-called \emph{Navier-Stokes equations}, given by
\begin{equation}\label{eq:Navier-Stokes}
    \dot{X} + \nabla_X X - \nu\Delta X = -\nabla p, \qquad \div(X) = 0,
\end{equation}
where $\Delta$ is the Hodge Laplacian and $\nu\geq 0$ represents the viscosity of the fluid. The Euler equations correspond to solutions of (\ref{eq:Navier-Stokes}) with viscosity $\nu = 0$. A long-standing open problem is whether smooth solutions to (\ref{eq:Navier-Stokes}) or (\ref{eq:Euler}) can develop finite-time singularities (blow-up) \cite{fefferman2006existence}. On the other hand, short-time existence and uniqueness of solutions are guaranteed by standard arguments of partial differential equations \cite{majda2003vorticity}.
\end{remark}

Euler equations can be understood as an infinite-dimensional (local) flow. Indeed, Equation (\ref{eq:Euler}) defines an infinite-dimensional dynamical system on the linear subspace $\Xdiv(M) \subseteq \fX(M)$ of divergence-free vector fields.

In this way, as in Section~\ref{S:uds}, it is interesting to study universality properties of this dynamical system to capture its complexity. The seminal result in this direction is the following theorem by~Tao.

\begin{theorem}[\cite{tao2017universality}]\label{thm:Tao-embedding}
    Let $B: \mathbb{R}^n \times \mathbb{R}^n \to \mathbb{R}^n$ be a bilinear map satisfying $\langle B(x,x),x \rangle = 0$ for all $x \in \mathbb{R}^n$.
Then, there exists a Riemannian manifold \(M\) and an embedding
\[
\varphi: \RR^n \to \Xdiv(M)
\]
such that if \(u(t)\) is any solution of the ordinary differential equation
\[
\dot{u} = B(u, u),
\]
then \(\varphi(u(t))\) is a solution to the Euler equations on \(M\).
\end{theorem}

\begin{remark}
    In the language of Section~\ref{S:uds}, Theorem \ref{thm:Tao-embedding} can be alternatively stated as that Euler flows are universal for quadratic flows $u$ in $\RR^n$ satisfying the vanishing condition $\langle B(u,u),u\rangle = 0$.
\end{remark}

The proof of Theorem \ref{thm:Tao-embedding} in \cite{tao2017universality} is fully constructive and demonstrates that $\varphi: \RR^n \to \Xdiv(M)$ is a linear embedding. The target manifold is chosen to be $M = \textup{SO}(n) \times \TT^n$, and the author constructed (time-independent) vector fields $X_1, \ldots, X_n \in \Xdiv(\textup{SO}(n) \times \TT^n)$ of the form
$$
    X_i = \xi_i + \sum_{i=1}^n a_i\partial_{t_i},
$$
where $\xi_i$ is the lift to $\textup{SO}(n) \times \TT^n$ of a right $\textup{SO}(n)$-invariant vector field on $\textup{SO}(n)$, $a_i \in \RR$ are suitably chosen constants, and $\partial_{t_1}, \ldots, \partial_{t_n}$ is the standard frame of the tangent bundle of $\TT^n$. These vector fields turn out to be Euler flows for a handcrafted metric and pressure function, showing that the desired embedding $\varphi: \RR^n \to \Xdiv(\textup{SO}(n) \times \TT^n)$ is given by $\varphi(x_1, \ldots, x_n) = x_1X_1 + \ldots + x_nX_n$. For details, please check \cite{tao2017universality}.

\begin{remark}
    The hypothesis $\langle B(x,x), x \rangle = 0$ in Theorem \ref{thm:Tao-embedding} is very natural. A classical solution to the Euler equations satisfies the so-called conservation of energy. Explicitly, given an Euler flow $X$, we define the energy function
    $$
        E(t) = \frac{1}{2}\int_M ||X_t||^2_g\,\mu_g,
    $$
    where $\mu_g$ is the Riemannian volume form. Then, using the Euler equation we have
    \begin{align*}
        \frac{d}{dt} E(t) &= \int_M \langle \dot{X},X\rangle_g\,\mu_g = -\int_M \langle \nabla_X X,X\rangle_g\,\mu_g - \int_M \langle \nabla p,X\rangle_g\,\mu_g \\
        &= -\int_M \left\langle\nabla\left(\frac{1}{2}||X||_g^2+ p\right),X\right\rangle_g\,\mu_g.
    \end{align*}
    Now, let us observe that $\div(fY) = \langle \nabla f, Y\rangle_g + f\div(Y)$ for any function $f \in C^\infty(M)$ and vector field $Y \in \fX(M)$. In particular, if $X$ is divergence-free, then $\langle \nabla f, X\rangle_g = \div(fX)$, and the last integral vanishes by Stokes theorem. In this context, if $X = \varphi(x)$, with $\dot{u} = B(u,u)$ and the embedding is isometric, the conservation of energy for $X$ is implied by the vanishing condition $\langle B(u,u), u\rangle = \frac{1}{2}\frac{d}{dt} ||u||^2 = 0$.
\end{remark}

This result has astonishing consequences on the dynamics of Euler flows, as shown in the following result.

\begin{corollary}[Torres de Lizaur \cite{lizaur2022chaos}]\label{cor:chaos-Euler}
    Euler flows are universal for vector fields on { the sphere} $\SS^n$ that extend to polynomial vector fields on $\RR^n$. In particular, there exists a Riemannian manifold $M$ for which the Euler flow in $\Xdiv(M)$ is chaotic. 
\end{corollary}

The proof of Corollary \ref{cor:chaos-Euler} roughly proceeds by embedding such a vector field $u$ on $\SS^n$ as the velocity field of a quadratic flow in $\RR^N$ that satisfies the aforementioned vanishing condition. This is done by expanding $u$ using spherical harmonics on $\SS^n$, see \cite{lizaur2022chaos} for details. In this setting, the condition that $u$ is tangent to $\SS^n$ corresponds precisely to the energy conservation law.


\subsection{Dual formulation of the Euler equation}

As we will exploit in the upcoming sections, there is a dual formulation of the Euler equation (\ref{eq:Euler}) in terms of dual forms that is in many cases more suitable for geometric arguments. In this sense, let $(X, p)$ be an Euler flow on a Riemannian manifold $(M, g)$ and let $\alpha = g(X, -)$ be the dual $1$-form to $X$.

Let us start by computing the Lie derivative $L_X \alpha$. To do so, let $Y$ be any vector field. We have
$$
    (L_X\alpha)(Y) = L_X(\alpha(Y)) - \alpha(L_XY) = X\langle X, Y\rangle - \langle X, [X,Y]\rangle.
$$
Now, using that $[X,Y] = \nabla_X Y - \nabla_Y X$, and $ X\langle X, Y\rangle = \langle\nabla_X X, Y\rangle + \langle X, \nabla_X Y\rangle$, we have that
\begin{align*}
    (L_X\alpha)(Y) &= \left[\langle\nabla_X X, Y\rangle + \langle X, \nabla_X Y\rangle\right] - \left[\langle X, \nabla_X Y\rangle - \langle X, \nabla_Y X\rangle\right] \\
    &= \langle\nabla_X X, Y\rangle + \langle X, \nabla_Y X\rangle = \langle\nabla_X X, Y\rangle + \frac{1}{2}d\left(||X||^2\right)(Y).
\end{align*}

On the other hand, using Cartan's formula, we get $L_X \alpha = \iota_Xd\alpha + d\iota_X\alpha = \iota_Xd\alpha + d\left(||X||^2\right)$. Therefore, we have
$$
    \langle\nabla_X X, -\rangle = L_X\alpha - \frac{1}{2} d\left(||X||^2\right) = \iota_Xd\alpha + \frac{1}{2}d\left(||X||^2\right).
$$
Therefore, dualizing (\ref{eq:Euler}), we get
$$
    0 = \langle\dot{X}, -\rangle + \langle\nabla_X X, -\rangle + \langle\nabla p, -\rangle = \dot{\alpha} + \iota_Xd\alpha + \frac{1}{2}d\left(||X||^2\right) + dp = \dot{\alpha} + \iota_Xd\alpha + d\left(p+\frac{1}{2}||X||^2\right).
$$
Hence, if we define the so-called Bernoulli function $B = p+\frac{1}{2}||X||^2$, the dual formulation of (\ref{eq:Euler}) is
\begin{equation}\label{eq:Euler-dual}
    \dot{\alpha} + \iota_Xd\alpha = - dB.
\end{equation}
This equation is known as the \emph{dual Euler equation}. The divergence-free condition for $X$ can be dually written as $d\star\alpha = 0$.

In the following, we will be also interested in stationary solutions to the Euler equation, i.e., time-independent solutions.

\begin{definition}
    A vector field $X$ is a \emph{stationary Euler flow}, or a \emph{steady Euler flow}, if it satisfies the equations
    \begin{equation}\label{eq:stationary-euler}
        \iota_X d\alpha = -dB, \qquad \div(X) = 0,
    \end{equation}
    where $\alpha = g(X, -)$ is the dual form and $B \in C^\infty(X)$ is the Bernoulli function.
\end{definition}

\subsection{The contact mirror}

One of the key advances of the dual formulation of the Euler equations is that it allows us to relate Euler flows to geometric structures on $M$. The most important instance of this link is the so-called \emph{contact mirror}, as unveiled by Sullivan and ellaborated by Etnyre and Ghrist \cite{etnyre1997contact}. 

Throughout this section, we shall suppose that $M$ has odd dimension $\dim M = 2n+1$. If $M$ is endowed with a Riemannian metric $g$, let us denote by $\mu$ the associated Riemannian volume. In this setting, given a vector field $X \in \fX(M)$ with dual form $\alpha = g(X, -)$, its curl is the unique vector field $\curl(X) \in \fX(M)$ satisfying $\iota_{\curl(X)} \mu = (d\alpha)^n$.

\begin{definition}
    Given a Riemannian manifold $(M,g)$, a \emph{Beltrami vector field} is a vector field $X \in \fX(M)$ satisfying
    $$
        \curl(X) = fX,
    $$
    for some $f\in C^\infty(M)$. If $f(p) > 0$ for all $p \in M$, we say that $X$ is a \emph{rotational Beltrami field}.
\end{definition}

\begin{remark}
    The Beltrami condition is equivalent to $(d\alpha)^n = f \iota_X \mu$. For the dual form $\alpha = g(X, -)$, this can be written as $(d\alpha)^n = f\star \alpha$.
\end{remark}

\begin{remark}
    In the case that $X$ is a Beltrami vector field with the property that $\curl(X) = \lambda X$ with $\lambda \in \RR$ and $\lambda \neq 0$, i.e., $X$ is an eigenfunction of $\curl$ with non-zero eigenvalue $\lambda$, then $X$ is automatically divergence-free, since
    $$
        0 = \div(\curl(X)) = \div(\lambda X) = \lambda\, \div(X).
    $$
    However, in general a Beltrami field for a non-constant function $f$ may not be divergence-free.
\end{remark}

Beltrami fields are significant because they correspond to geometric structures on $M$ known as contact forms.

\begin{definition}
Let $M$ be a manifold of dimension $2n+1$. A \emph{contact form} for $M$ is a $1$-form $\alpha \in \Omega^1(M)$ such that $\alpha \wedge (d\alpha)^n \neq 0$ everywhere. A vector field $X$ satisfying $\iota_X d\alpha = 0$ is called a \emph{Reeb-like} vector field.
\end{definition}


\begin{remark}
    A contact form defines a $2n$-plane distribution $\xi = \ker\alpha$, called the \emph{contact structure}. The contact condition is equivalent to the fact that $\xi$ is maximally non-integrable. Furthermore, this implies that $d\alpha$ defines a symplectic form on $\xi$. However, a contact structure only defines the contact form up to rescaling so, in particular, it does not determine the Reeb-like vector fields.
\end{remark}

\begin{remark}
    Since $d\alpha$ is non-degenerate on $\xi = \ker\alpha$, this implies that Reeb-like vector fields form a line bundle transverse to $\xi$. Among them, there is a particular one $R$ satisfying additionally that $\alpha(R) = 1$, called the \emph{Reeb vector field}. 
\end{remark}

A theorem proved by Etnyre and Ghrist in dimension $3$ in \cite{etnyre1997contact}, and later extended to higher dimensions by the last two authors of this work jointly with Cardona and Presas \cite{CMPP2}, shows that these seemingly unrelated notions are actually strongly linked.

\begin{theorem}[\cite{etnyre1997contact,CMPP2}]\label{thm:contact-mirror}
    Let $M$ be an odd-dimensional orientable manifold and $X$ a nowhere vanishing vector field on $M$. If there exists a Riemannian metric for which $X$ is a rotational Beltrami field, then it is a Reeb-like vector field of some contact form on $M$. Reciprocally, if there exists a contact form for which $X$ is a Reeb-like vector field, then $X$ is a rotational Beltrami vector field for a certain Riemannian metric. Such metric can be additionally chosen so that $X$ is divergence-free.
\end{theorem}

\begin{proof}
    Suppose that $X$ is a nowhere vanishing Beltrami field on a Riemannian manifold $(M, g)$. Then, taking $\alpha = g(X, -)$, we have
    $$
        \alpha \wedge (d\alpha)^n = \alpha \wedge \iota_{\curl(X)} \mu = f \alpha \wedge \iota_X \mu = f ||X||^2 \mu \neq 0,
    $$
    where in the last equality we have used that $0 = \iota_X(\alpha \wedge \mu) = \alpha(X) \mu - \alpha \wedge \iota_X \mu = ||X||^2\mu - \alpha \wedge \iota_X$. This proves that $\alpha$ is a contact form. Now, observe that $\iota_X (d\alpha)^n = \iota_X \iota_{\curl(X)} \mu = 0$ and thus, since $d\alpha$ is non-degenerate on $\ker(\alpha)$, this implies that $\iota_X d\alpha = 0$, so $X$ is Reeb-like.

    Reciprocally, suppose that $\alpha$ is a contact form and $X$ is Reeb-like. Let us pick a metric $g$ on $M$ such that $\alpha = g(X, -)$. Then, such a metric converts $X$ into a Beltrami field since $\iota_X \iota_{\curl(X)} \mu = \iota_X (d\alpha)^n = n \iota_X d\alpha \wedge (d\alpha)^{n-1} = 0$, so $X$ and $\curl(X)$ are parallel. Furthermore, if the metric $g$ is chosen additionally satisfying that $\xi = \ker \alpha$ is orthogonal to $X$ and the Riemannian volume on $\xi$ agrees with the symplectic volume $(d\alpha)^n$, then $\iota_X \mu = (d\alpha)^n$ and thus $d\iota_X\mu = 0$ so $X$ is divergence-free.
\end{proof}

\begin{remark}
    If $\alpha$ is a contact form with Reeb vector field $R$, the metrics satisfying that $\alpha = g(R, -)$, $\xi = \ker \alpha$ is orthogonal to $R$, and $g|_{\xi}$ is compatible with the symplectic structure induced by $d\alpha$ on $\xi$, are called \emph{compatible metrics}. In Theorem \ref{thm:contact-mirror}, if $X$ is the Reeb vector field, then any compatible metric makes $X$ divergence-free and thus the metric used in Theorem \ref{thm:contact-mirror} is not unique.
\end{remark}

Theorem \ref{thm:contact-mirror} is particularly significant because Beltrami fields are actually special solutions of the Euler equations.

\begin{proposition}\label{prop:Beltrami-Euler}
    If $X$ is a divergence-free rotational Beltrami vector field, then it is a stationary solution to the Euler equation with constant Bernoulli function.
\end{proposition}

\begin{proof}
    Since $X$ is a Beltrami field, for the dual form $\alpha = g(X, -)$ we have
    $$
        \iota_X(d\alpha)^n = \iota_X\iota_{\curl(X)}\mu = f\iota_X\iota_{X}\mu = 0. 
    $$
    Now, since $d\alpha$ is maximally non-degenerate by Theorem \ref{thm:contact-mirror}, then $\iota_X (d\alpha)^n = 0$ implies $\iota_X d\alpha = 0$. This exactly means that $X$ satisfies the stationary Euler equation (\ref{eq:stationary-euler}) with $dB = 0$.
\end{proof}

\begin{remark}
    In dimension higher than $3$, it is not generally true that a non-rotational Beltrami vector field is a stationary Euler flow. The reason is that, in that case, $d\alpha$ might not be non-degenerate and thus $\iota_X(d\alpha)^n = 0$ does not imply that $\iota_X d\alpha = 0$. A counterexample can be found in \cite[Theorem 4.3]{cardona2021steady}. However, if $d\alpha$ is generically non-degenerate, then $\iota_X d\alpha = 0$ generically and thus everywhere by continuity (cf.\ \cite[Proposition 4.7]{cardona2021steady}).
\end{remark}

\begin{remark}
    If $X$ is an Euler flow, the vector field $\curl(X)$ is called the \emph{vorticity} of the flow and plays an important role in the study of the solutions \cite{judovivc1963non,majda2003vorticity,zlatovs2015exponential}. In this language, a Beltrami field is precisely one that is collinear with its vorticity.
\end{remark}

\begin{corollary}\label{cor:mirror-solutions}
    Any vector field that is Reeb-like for some contact form is a stationary Euler flow for some metric.
\end{corollary}

Corollary \ref{cor:mirror-solutions} has several striking consequences derived from contact topology:
\begin{itemize}
    \item In dimension $3$, every orientable manifold admits a contact structure by Martinet's theorem \cite{martinet1971formes}, and thus admits a non-vanishing stationary Euler flow (in fact, it admits infinitely many ones).
    \item For every knot $K \subseteq \SS^3$, there exists an Euler flow in $\SS^3$ with a closed orbit equivalent to $K$. It was originally proved in \cite{ghrist1997branched} for some adapted metric, and in \cite{enciso2012knots} for the round metric and with vortex tubes (solving the so-called Kelvin's conjecture).
    \item Furthermore, there exists an Euler flow in $\SS^3$ (for some adapted metric) with closed orbits of all knot types simultaneously \cite{etnyre2000contact}.
\end{itemize}

\begin{remark}
    Theorem \ref{thm:contact-mirror} has an analogue in the realm of cosymplectic geometry, given by a pair of closed forms $\alpha \in \Omega^1(M)$ and $\omega \in \Omega^2(M)$ such that $\alpha \wedge \omega$ is a volume form. In this setting, in \cite{dyhr2025turing} we show that cosymplectic structures are in correspondence with harmonic vector fields. These harmonic vector fields lie in the kernel of the $\curl$ operator and not only produce stationary solutions to the Euler equations but also provide stationary solutions to the full Navier-Stokes equations.
\end{remark}

\section{Computability in Dynamical Systems}

\subsection{Universality for discrete dynamical systems}

A discrete dynamical system is a pair $(A, f)$, where $A$ is a manifold and $f$ is a bijective function $f: A \to A$, not necessarily continuous. Analogously to the continuous case, the orbit of a point $a \in A$ is the collection $\{f^n(a)\}_{n \in \ZZ}$. However, in this case it is not straightforward to define what it means to embed such a discrete system into a continuous one. To address this, we introduce the notion of a Poincar\'e first return map.

\begin{definition}
    Let $M$ be a manifold equipped with a vector field $X$, whose associated flow we denote by $\varphi_t$. A \emph{Poincar\'e section} is a submanifold $A \subseteq M$ of codimension $1$, possibly with boundary, satisfying the following conditions:
    \begin{itemize}
        \item $A$ is transverse to the flow of $X$, i.e., $X_a$ is not tangent to $A$ for any $a \in A$.
        \item $A$ is non-wandering, i.e., for every $a \in A$ we have that $\varphi_t(a) \in A$ for some $t > 0$. 
    \end{itemize} 
\end{definition}

Observe that, since $X$ is transverse to the Poincar\'e section $A$, the flow $\varphi_t(a)$ of $a \in A$ intersects $A$ at discrete times. In particular, it makes sense to define the first return time $t_a$ as
$$
    t_a = \min\left\{t > 0 \,\mid\, \varphi_t(a) \in A\right\}.
$$

\begin{definition}
    Given a manifold $M$ endowed with a vector field $X$ with flow $\varphi_t$, and a Poincar\'e section $A \subseteq M$, the \emph{Poincar\'e first return map} is the function $P_A: A \to A$ given by $P_A(a) = \varphi_{t_a}(a)$.
\end{definition}


\begin{definition}
    An \emph{embedding} of a discrete dynamical system $f: A \to A$ into a continuous dynamical system $(M, X)$ is an embedding of smooth manifolds $\varphi: A \to M$ such that $\varphi(A) \subseteq M$ is a Poincar\'e section and $\varphi \circ f = P_A \circ \varphi$.
\end{definition}

\begin{proposition}\label{prop:mapping-torus}
    If $f: A \to A$ is a smooth diffeomorphism, then $(A, f)$ can be embedded into the mapping torus of $f$.        
\end{proposition}

\begin{proof}
    Define the mapping torus of $f$ as $M_f = A \times [0,1]/(a, 0) \sim (f(a), 1)$. Then $A \times \{0\}$ is a Poincar\'e section for the vertical vector field $\partial_t$, and the associated Poincar\'e return map is exactly the map $f$.
\end{proof}

These embeddings of discrete into continuous dynamical systems allow us to define an analogous notion of universality in this setting.

\begin{definition}
    Let $\mathcal{C} = \{(A, f)\}$ be a collection of discrete dynamical systems. We will say that a class $\mathcal{U} = \{(M, X)\}$ of continuous dynamical systems is \emph{universal} for $\mathcal{C}$ if for every $(A, f) \in \mathcal{C}$, there exists $(M, X) \in \mathcal{U}$ and an embedding of $(A, f)$ into $(M, X)$.
\end{definition}

In this language, Proposition \ref{prop:mapping-torus} can be re-stated as that mapping tori provide a universal class for smooth discrete dynamical systems.

\subsection{Turing machines as discrete dynamical systems}

In this section, we shall show that a computer, or better said the mathematical formalization of it called a Turing machine, can be seen as a discrete dynamical system. Roughly speaking, a Turing machine is a read/write device that has access to a two-sided infinite tape as memory divided into cell memories. Each cell memory can be written with a symbol from a prefixed alphabet or left empty. This read/write device is represented by a head that is placed on top of one of the symbols of the tape. At each step, the head reads the symbol below and, using this information and its internal state, substitutes the symbol with another one, shifts the tape either left or right, and changes its internal state. This process is performed until the head reaches a halting state in which the computation stops, or continues indefinitely. Figure \ref{fig:TuringMachine} provides a pictorial representation of a Turing machine.

\begin{figure}[H]
    \centering
    \includegraphics[scale=0.4]{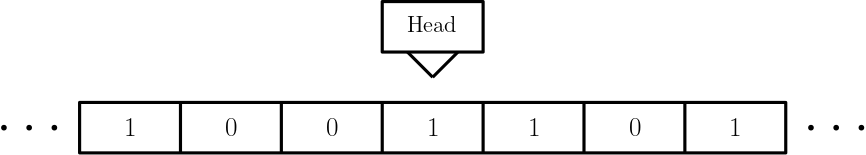}
    \caption{Representation of a Turing machine.}
    \label{fig:TuringMachine}
\end{figure}

This intuitive notion can be formalized as follows. Let us consider a non-empty finite set $A$, called the \emph{alphabet}. This alphabet contains a special symbol, denoted by $\square \in A$ and called the \emph{blank symbol}.

\begin{definition}[Adapted from {\cite{turing1936computable}}]
A \emph{Turing machine} $M$ with alphabet $A$ is a tuple $(Q, q_0, Q_{\mathrm{halt}}, \delta)$ where:
\begin{itemize}
\item $Q$ is a finite set, called the \emph{states} of $M$, and $q_0\in Q$ is known as the \emph{initial state}.
\item $Q_{\mathrm{halt}} \subsetneq Q$ is a subset called \emph{halting states}.
\item $\delta$ is a function
\[
  \delta : Q \times A \longrightarrow Q \times A \times \{-1,+1\},
\]
called the \emph{transition function}.
\end{itemize}
\end{definition}

The transition function $\delta: Q \times A \to Q \times A \times \{-1,+1\}$ represents the operation performed by the read/write head. Hence, if $\delta(q,a)=(q',a',s)$, this means that when the head reads the symbol $a \in A$ in the tape when at state $q \in Q$, the head substitutes it by $a' \in A$, shifts the tape by $s = \pm 1$ positions, and changes its internal state to $q' \in Q$. The content of a tape will be represented by an element in { the set $A^*$ of compactly supported two-sided sequences in $A$, i.e.,\ sequences $t = \{t_n\}$ with $t_n \in A$ satisfying that $t_m = \square$ for sufficiently large $|m|$}. In this manner, if $t = \{t_n\} \in A^*$ is the current tape and $a = t_0$ is the read symbol, the new tape will be $t' = \{t'_n\} \in A^*$ with $t_{n}'=t_{n+s}$ if $n \neq -s$ and $t_{-s} = a'$.

In this way, if we denote $T = Q \times A^*$, we have that a Turing machine $M$ defines a discrete dynamical system on $T$
$$
    \Delta_M: T \longrightarrow T.
$$
This is precisely the discrete dynamical system that we aim to represent inside a continuous dynamical system.

It turns out that there is a better way of representing this dynamical system. Let us consider $B = \{0,1\}$ {with $0$ being the blank symbol}, so that $B^*$ is the set of two-sided binary sequences with finitely many $1$'s. A \emph{generalized shift} is a pair $S=(r, G, F)$, where $r > 0$ is an integer, $G$ is a function $G: B^r \to B^r$ and $F$ is a function $F: B^r \to \ZZ$. This information defines a dynamical system
$$
    \Lambda_S: B^* \longrightarrow B^* 
$$
as follows. Let $t = \{t_n\} \in B^*$ and set $(b_0, \ldots, b_{r-1}) = G(t_0, \ldots, t_{r-1})$ and $s = F(t_0, \ldots, t_{r-1})$. Then, $\Lambda_S(t)$ is the sequence $t' = \{t_n'\}$ given by $t_n' = t_{n+s}$ for $n \not\in [-s, r-s+1]$ and $t_n' = b_{n+s}$ for $n \in [-s, r-s+1]$. In other words, $\Lambda_S(t)$ is the sequence where we have replaced $t_0, \ldots, t_{r-1}$ by $b_0, \ldots, b_{r-1}$ and then shift $s$ positions. A representation of this operation of a generalized shift is shown in Figure \ref{fig:generalized_shift}.

\begin{figure}[H]
    \centering
    \includegraphics[scale=0.45]{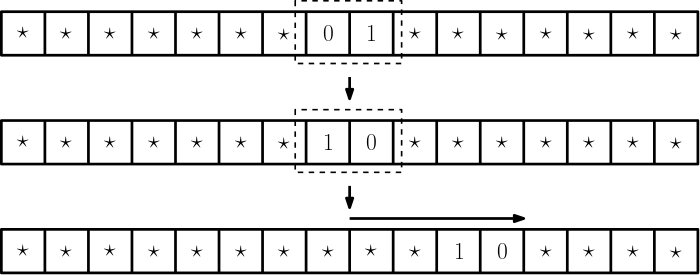}
    \caption{Representation of a generalized shift with $r = 2$, $G(0,1) = (1,0)$ and $F(0,1) = 3$.}
    \label{fig:generalized_shift}
\end{figure}

\begin{remark}
    The original definition of a generalized shift, as formulated in \cite{moore1991generalized}, includes the possibility that the ranges of action of $G$ and $F$ are different, and not starting at zero. However, this can be easily arranged by expanding the domains of $G$ and $F$ to a common domain, and then relabeling the positions of the sequence so that the starting point of the domain is the zero position.
\end{remark}

Coming back to Turing machines, observe that, since the states $Q$ and alphabet $A$ are finite, we can consider an injection $\varphi: Q \cup A \to B^*$ by encoding them in binary. This defines an injective map $\varphi^*: T = Q \times A^* \to B^*$ by sending $(q, t) \in T$ to the sequence that starts with $\varphi(q) \in B^*$ in the zero position, then places the substring $\{t_n\}_{n \geq 0}$ of $t$ to the right of $\varphi(q)$ and $\{t_n\}_{n < 0}$ to its left. In this direction, Moore identified the image of such embedding as a generalized shift.

\begin{theorem}[\cite{moore1991generalized}]\label{thm:TM-as-GS}
    Given a Turing machine $M$ with dynamical system $\Delta_M: T \to T$ and an embedding $\varphi: Q \cup A \to B^*$ that induces $\varphi^*: T \to B^*$, there exists a generalized shift $S$ on $B^*$ such that
    $$
       \varphi^* \circ \Delta_M = \Lambda_{S} \circ \varphi^*.
    $$
\end{theorem}

\subsection{Generalized shifts in the square Cantor set}\label{sec:cantor}

As seen in the previous section, the generalized shift map $\Lambda_S: B^* \to B^*$ provides a class of discrete dynamical systems that encode computational dynamics. Despite its explicit description, the base set $B^*$ of two-sided binary sequences is not a smooth manifold and, in particular, it is not clear how to represent it in the sense of Definition \ref{def:universality}.

Fortunately, $B^*$ can be naturally embedded in a manifold using a combinatorial trick. Let $C$ be the \emph{standard ternary Cantor set} $C \subset [0,1]$, which is constructed by removing the mid-segment of a division of the interval $I$ into three equal intervals and proceeding recursively. Explicitly,
$$
    C = \bigcap_{m=0}^\infty C_m, 
$$
where $C_0 = [0,1]$ and, for $m \geq 1$, the set $C_m$ is given recursively as
$$
    C_m = \frac{1}{3}\left(C_{m-1} \cup (2 + C_{m-1})\right).
$$
We can also consider the ternary $n$-th Cantor set $C^n \subseteq [0,1]^n$. For the purposes of this work, the square Cantor set $C^2 \subseteq [0,1]^2$ will play an important role.

Using the Cantor set, we can perform several interesting encodings as follows. 
\begin{enumerate}
    \item The Cantor set $C$ is the collection of real numbers $0 \leq x \leq 1$ with an expansion in base $3$ not containing a digit $1$, i.e., of the form $x = \sum_{n \geq 1} \epsilon_i 3^{-i}$, with $\epsilon_i = 0 $ or $2$ for every $i$. Using the previous description, the Cantor set $C$ can be identified with the collection of one-sided binary sequences.
    \item In the same vein, the square Cantor set $C^2$ can be identified with the collection of two-sided binary sequences. In fact, if $\{t_n\}_{n \in \ZZ}$ is a two-sided sequence, then we decompose it into the two one-sided sequences $t_+ = \{t_n\}_{n \geq 0}$ and $t_- = \{t_n\}_{n < 0}$. Encoding $t_+$ in the first factor of $C^2$ and $t_-$ in the second factor, we get the identification. In particular, we get a natural inclusion $\kappa: B^* \hookrightarrow C^2$ of the eventually zero two-sided sequences.
\end{enumerate}

By embedding the square Cantor set $C^2$ into the closed $2$-dimensional disc $\mathbb{D}$ of radius $2$, we obtain an embedding
$$
    \kappa: B^* \hookrightarrow \mathbb{D}.
$$
Under this embedding, sequences beginning with a fixed finite substring can be easily identified. Let $b = (b_0,\ldots, b_{r-1}) \in B^*$ be a finite word, and denote by $B_{b}^*$ the collection of sequences $t = \{t_n\}$ with $t_0 = b_0, t_1 = b_1, \ldots, t_{r-1} = b_{r-1}$. The subset $\kappa(B_b^*) \subseteq \DD$ is contained in a square block in $\DD$, as shown in Figure \ref{fig:cantor-block}.

\begin{figure}[H]
    \centering
    \includegraphics[scale=0.3]{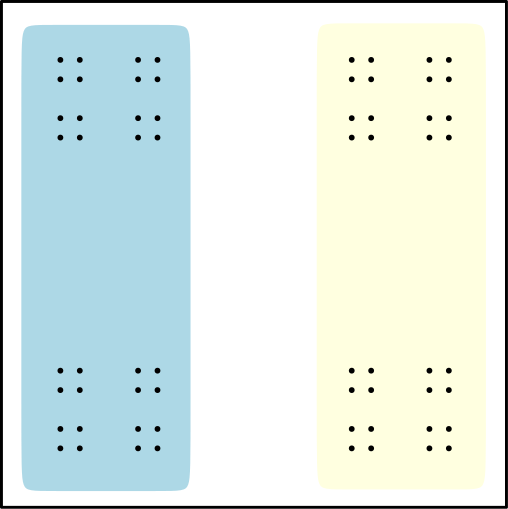}
    \caption{Cantor blocks corresponding to the two-sided finite binary sequences $B_{0}$ starting with $0$ (blue) and $B_{1}$ starting with $1$ (yellow).}
    \label{fig:cantor-block}
\end{figure}

This embedding allows us to extend bijective generalized shifts $\Lambda_S: B^* \to B^*$ to simple diffeomorphisms of the disc. They are determined by a finite set of words $b^1, \ldots, b^n \in B^*$ of the same length and a permutation $\sigma: \{1, \ldots, n\} \to \{1, \ldots, n\}$. Send linearly the blocks $\kappa(B_{b^k}^*) \subseteq \DD$ into the blocks $\kappa(B_{\sigma(k)}^*) \subseteq \DD$, and extend the map to the whole disc in such a way that the map on the boundary of $\DD$ is the identity. Observe that all the involved blocks have the same volume, so the extension can be done volume-preserving. Hence, this gives rise to a volume-preserving diffeomorphism $f: \mathbb{D} \to \mathbb{D}$ that is the identity on the boundary, called a \emph{block diffeomorphisms} of the disc.

\begin{theorem}[\cite{moore1991generalized,CMPP}]\label{thm:GS-as-block}
    For every bijective generalized shift $S$, there exists a block diffeomorphism $f$ of the disc such that
    $$
         \kappa \circ \Lambda_S = f \circ \kappa.
    $$
\end{theorem}

\subsection{Turing-universality in Euler flows}

By Theorem \ref{thm:TM-as-GS}, generalized shifts represent the dynamical system of a Turing machine, which are embedded into the disc as block diffeomorphisms by Theorem \ref{thm:GS-as-block}. These observations lead to the following key definition.

\begin{definition}\label{defn:Turing-universal}
    A class $\mathcal{U} = \{(M, X)\}$ of continuous dynamical systems is said to be \emph{Turing-universal} if it is universal for the class $\mathcal{C} = \{(\mathbb{D}, f)\}$ of block diffeomorphisms of the disc.
\end{definition}

Few known examples exist of dynamical systems that are Turing-universal. One of the most important examples was proven by two of the authors of this work jointly with Cardona and Presas.

\begin{theorem}[\cite{CMPP}]\label{thm:Euler-universal}
    Let $M$ be an orientable $3$-dimensional manifold. Then, the Euler flows on $M$ are Turing-universal.
\end{theorem}

\begin{proof}[Sketch of proof]
The construction proceeds through the following steps.
\begin{itemize}
    \item Let $f: \DD \to \DD$ be a block diffeomorphism of the disc. Consider the mapping torus $M_f = \DD \times [0,1]/(x,0) \sim (f(x), 1)$ associated to $f$ endowed with the vertical vector field $\partial_t$. The Poincaré first return map of $\partial_t$ on the slice disc $\DD \times \{0\}$ is equal to $f$.
    \item The diffeomorphism $f$ is isotopic to the identity, so $M_f$ is diffeomorphic to a standard solid torus $\DD \times \SS^1$. Hence, by pushing the vector field $\partial_t$ through this diffeomorphism, we get a vector field $X_0$ on $\DD \times \SS^1$ so that the Poincar\'e first return map of $X_0$ is $f$ on $\DD \times \{0\}$. Observe that since $f$ is the identity around the boundary of $\mathbb{D}$, $X_0 = \partial_t$ in an open set containing the boundary $\partial\DD \times \SS^1$.
    \item Since $f$ is volume-preserving on $\DD$, it is symplectic with respect to the standard symplectic form. In this way, we can construct a contact form $\alpha$ on $\DD \times \SS^1$ in such a way that its Reeb vector field $X$ coincides with $X_0$ in an open set around $\DD \times \{0\}$ and $\partial \mathbb{D} \times [0,1]$. The Poincar\'e first return map of $X$ is still $f$.
    \item Choose an embedded loop $\gamma$ in $M$, whose tubular neighbourhood is $\DD\times\SS^1$. Extend the contact form $\alpha$ from this tubular neighbourhood to the whole manifold $M$, and denote by $X$ the corresponding Reeb vector field.
    \item By the contact mirror, Theorem \ref{thm:contact-mirror}, there exists an adapted Riemannian metric on $M$ such that $X$ is a divergence-free Beltrami field. Since $X$ agrees with the original vector field on the tubular neighbourhood of $\gamma$, the image of $\mathbb{D}\times\{0\}$ on $M$ is a Poincar\'e section for $X$ and the Poincar\'e first return map is $f$. Therefore, by Proposition \ref{prop:Beltrami-Euler}, $X$ is an Euler flow, as we wanted.
\end{itemize}
\end{proof}

\begin{remark}
    Theorem \ref{thm:Euler-universal} shows universality of Euler flows on $M$ \emph{for some Riemannian metric}. This metric is conditioned by the contact structure constructed in the process, so it depends on the Turing machine to be represented.
\end{remark}

\begin{remark}
    As proved in \cite{CMPP}, the proof of Theorem \ref{thm:Euler-universal} can be actually improved so that if $M$ is initially equipped with a contact structure $\alpha$, then the contact structure can remain unchanged outside the tubular neighbourhood after the surgery. 
    Notably, the metric associated with $\alpha$ remains unchanged outside this neighbourhood throughout the construction.
\end{remark}

{
We finish this section by emphasizing that the notion of Turing-universality discussed in this paper is not the only one existing in the literature. Alternative ways of ``embedding'' a computation into a dynamical system have been reported in the literature.

In this direction, one of the most promising approaches is to characterize Turing machines through solutions of ordinary differential equations, such as in \cite{bournez2023continuous, blanc2025simulation}. In this setting, the main characters are solutions to an initial value problem of the form
\begin{equation}\label{eq:ode-turing}
    \frac{\partial y}{\partial x} = P(y), \qquad y(x_0) = y_0, 
\end{equation}
for a function $y = y(x): \mathbb{R} \to \mathbb{R}^m$, with a given polynomial $P(y)$ and initial condition $(x_0, y_0) \in \RR \times \RR^m$. In the aforementioned works, the authors showed that any computable function can be obtained as (a truncation of) a solution of the system (\ref{eq:ode-turing}) for a certain polynomial, and they actually characterize the complexity class PSPACE by using this approach.

The idea of representing computation by means of physical systems has also been very recurrent in the literature. In \cite{beggs2007can}, the authors proposed a physical device, a quarter-disc marble run, able to represent the characteristic function of any subset $A \subseteq \NN$. To do so, they used infinitely many trays and place springs at the end of the tray. If the $n$-th marble bounces at the spring and come back to the initial position, then $n \in N$, and if there is no spring in the $n$-th tray then $n \not\in \NN$. In this direction, more exotic systems have been proposed, such as the Hawking radiation of a black hole \cite{etesi2002non,andrews2019black}, as a source of information capable of representing the output of Turing machines.

There also exist approaches to this physical representation of computation by using chemical reactions, mainly by encoding Turing-complete cellular automata \cite{wolfram1983statistical,wolfram2003new}, starting from the seminal paper \cite{wolfram2019cellular}. In recent years, these ideas have gained attraction due to the so-called surface chemical reaction networks \cite{qian2010efficient,yu2024cellular}, which are bi-dimensional latices of molecules that can perform either chemical reactions associated with one of the nearest molecules or uni-molecular reactions autonomously, thus capturing the essence of the Turing-universal Rule 110 in cellular automata. Additionally, this Turing-universality has reached biological systems, by representing computation by means of DNA folding \cite{jung2025test,sidl2025computational}.
}












\section{Topological Kleene Field Theories}

\subsection{The Church-Turing thesis}

One of the most fascinating aspects of computability is its diverse representations. Over the years, a number of definitions of computability have been proposed in the literature. The most celebrated one is as the process performed by a Turing machine \cite{turing1936computable}, but many other alternative approaches have been proposed, such as $\lambda$-calculus \cite{church1932set}, combinatory logic \cite{curry1930grundlagen}, or counter machines \cite{shepherdson1963computability}, among others.

However, the efforts of several mathematicians and logicians over the last century have shown that all these models turn out to be `equivalent', in the sense that they can perform the same calculations, despite their seemingly completely different nature. This striking feature was already recognized by Turing, who in \cite{turing1939systems} stated that any effectively calculable function, in the sense that it is computable in any of the aforementioned senses, must be calculable by means of a Turing machine.
This statement, known as the `Church-Turing thesis', posits that all possible models of computability are equivalent. Notice that this assertion is not a conjecture nor a theorem since it cannot be formally proven as it depends on our definition of computability. Rather, it is a meta-conjecture—a `thesis' in the philosophical sense.

In this way, depending on the perspective we want to approach computability, the use of alternative models of computation can be convenient, since they enable new viewpoints. Among these models, the model of partial recursive functions, introduced by Kleene in \cite{kleene1936general}, is particularly effective for representing actual computations in dynamical systems.

Recall that a \emph{partial function} $f$ between sets $A$ and $B$ is a function $f: D_f \to B$ defined on a certain subset $D_f \subseteq A$ called the domain of $f$. Equivalently, any partial function can be extended to a total function by augmenting its domain. We do it as follows: choose a new element $\bullet$, which will represent the `undefined value', set $A^* = A \cup \{\bullet\}$ and $B^* = B \cup \{\bullet\}$, and define the usual function $\tilde{f}: A^* \to B^*$ by $\tilde{f}(a) = f(a)$ if $a \in D_f$ and $\tilde{f}(a) = \bullet$ if $a \not\in D_f$. It is customary to denote a partial function between $A$ and $B$ as $f: A \dashrightarrow B$. Observe that partial functions can be composed after restricting them to a common domain.

Among these partial functions, there exists a distinguished subset called the set of \emph{partial recursive functions}, also known as \emph{general recursive functions} or \emph{$\mu$-recursive functions}. It is defined as the smallest set of partial functions satisfying the following properties.

\begin{enumerate}
    \item For all $n_0 \in \NN$, the constant function $C_{n_0}: \NN^n \to \NN$, $C_{n_0}(x_1,\ldots, x_n) = n_0$, is a partial recursive function.
    \item The successor function $S: \NN \to \NN$, $S(x) = x + 1$, is a partial recursive function.
    \item The projection functions $\pi_i: \NN^n \to \NN$, $\pi_i(x_1, \ldots, x_n) = x_i$, are partial recursive functions for $i = 1, \ldots, n$.
    \item A partial function $f: \NN^n \dashrightarrow \NN^m$, $f(x_1, \ldots, x_n) = (f_1(x_1, \ldots, x_n), \ldots, f_m(x_1, \ldots, x_n))$, is a partial recursive function if and only if $f_1, \ldots, f_m: \NN^n \dashrightarrow \NN$ are partial recursive functions.
    \item\label{item:composition-PRF} If $f: \NN^n \dashrightarrow \NN^m$ and $g: \NN^m \dashrightarrow \NN^r$ are partial recursive functions, then $g \circ f: \NN^n \dashrightarrow \NN^r$ is a partial recursive function.
    \item Given partial recursive functions $f: \NN^n \to \NN$ and $g: \NN^{n+2} \to \NN$, the unique function $h: \NN^{n+1} \dashrightarrow \NN$ satisfying
    $$
        h(0, x_1, \ldots, x_n) = f(x_1, \ldots, x_n),\quad h(y + 1, x_1, \ldots, x_n) = g(y, h(y, x_1, \ldots, x_n), x_1, \ldots, x_n),
    $$
    is a partial recursive function.
    \item Let $f: \NN^{n+1} \dashrightarrow \NN$ be a partial recursive function. Given $(x_1, \ldots, x_n) \in \NN^n$, consider the set
    $$
        A_{(x_1, \ldots, x_n)} = \{y \,|\, f(y, x_1, \ldots, x_n) = 0 \textrm{ and } [0, y) \times \{(x_1,\ldots, x_n)\} \subseteq D_f\}.
    $$
    Then, the function $g: \NN^n \dashrightarrow \NN$ given by 
    $$
        g(x_1, \ldots, x_n) = \left\{ \begin{array}{ll} \min A_{x_1, \ldots, x_n} & \textrm{if } A_{x_1, \ldots, x_n} \neq \emptyset, \\
        \bullet & \textrm{otherwise},\end{array}\right.
    $$
    is a partial recursive function.
\end{enumerate}

\begin{remark}
    The previous set of axioms should be understood as follows. Properties (1)-(4) state that the simplest functions are indeed partial recursive functions, whereas axioms (5)-(6) just say structural properties of partial recursive functions, such that being partial recursive can be checked component-wise or that they are closed under composition. The important properties are (7), which state that functions defined by recursion are partial recursive, and (8), which says that finding the minimum zero of a function is partial recursive.
\end{remark}

\begin{remark}
    The functions generated by only applying axioms (1)-(7) are actually total functions, defined on the whole space, and define a proper subset known as primitive recursive functions. Only the minimization function of (8) is partial, since it is undefined in the case that $f$ has no zeroes.
\end{remark}

\begin{remark}
    Recall that there exists a natural embedding $\iota: \NN^n \to \NN$ given by $\iota(x_1, \ldots, x_n) = p_1^{x_1} p_2^{x_2} \cdots p_n^{x_n}$, where $p_i$ denotes the $i$-th prime number. This embedding can be easily proven to be a (total) recursive function. In this way, through this embedding, we can restrict our attention to partial recursive functions of the form $f: \NN \dashrightarrow \NN$.
\end{remark}

A key property of partial recursive functions is that they provide an equivalent model of computation, as posited by the Church-Turing thesis and proven by S.\ Kleene.

\begin{theorem}[\cite{kleene1936general}]\label{thm:kleene}
    Partial recursive functions are exactly the functions that can be computed with a Turing machine.
\end{theorem}

Explicitly, Theorem \ref{thm:kleene} establishes the following equivalence. If $f: \NN \dashrightarrow \NN$ is a partial recursive function, then there exists a Turing machine $M$ such that, when executed with input the number $n \in \NN$, binary encoded, then $M$ halts if and only if $f(n) \neq \bullet$, in which case the output on the tape is precisely $f(n)$. Reciprocally, given any Turing machine $M$, the function $f: \NN \dashrightarrow \NN$ given by
$$
    f(n) = \left\{\begin{array}{ll}
        m & \textrm{if $M$ halts when $n$ is used as initial value and returns $m$}, \\
        \bullet &  \textrm{if $M$ does not halt when $n$ is used as initial value},
    \end{array}\right. 
$$
is a partial recursive function.

\begin{remark}
    The correspondence of Theorem \ref{thm:kleene} is far from being bijective: The same partial recursive function can be computed by several different non-equivalent Turing machines.
\end{remark}


\subsection{Dynamical bordisms}

The framework of partial functions opens a new perspective in the representation of discrete dynamics in continuous systems. This viewpoint is explored in the work \cite{GPMPS}, which provided a novel interpretation of computability in dynamics.

Suppose that $M$ is a smooth manifold with a vector field $X$. Consider two submanifolds $A_1, A_2 \subseteq M$ transverse to the flow of $X$. This setting naturally defines a partial function
$$
    \rho_{(M, X)}: A_1 \dashrightarrow A_2,
$$
called the \emph{reaching function}, as follows. Let $a \in A_1$ and consider its flow $\varphi_t(a)$ under $X$. If $\varphi_t(a) \in A_2$ for some $t > 0$, then we set $\rho_{(M, X)}(a) = \varphi_{t_a}(a)$, where $t_a > 0$ again denotes the minimum positive time in which the flow hits $A_2$. Otherwise, if the flow $\varphi_t(a)$ skips $A_2$ for all $t > 0$, we set $\rho_{(M, X)}(a) = \bullet$.

A particularly interesting case of reaching functions occurs in the context of \emph{dynamical bordisms}, which provide a natural way to model computational processes using smooth manifolds. These are manifolds $M$ with boundary $\partial M = A_1 \sqcup A_2$ endowed with a vector field $X$ which is transverse to the boundaries and points inward at $A_1$ and outward at $A_2$. In this manner, the flow of a point starting at the incoming boundary $A_1$ can either reach the boundary $A_2$ at a single point or get trapped inside $M$ and never get out of the bordism. For a detailed formulation of these dynamical bordisms, see \cite[Section 3]{GPMPS}.

\begin{remark}
    For applications, we will typically need to enlarge the class of dynamical bordisms to allow manifolds with corners and with a bigger boundary than just $A_1$ and $A_2$. In this case, the flow $X$ will be required to be transverse and inward pointing at $A_1$, transverse and outward pointing at $A_2$ and parallel to the remaining part of the boundary, as shown in Figure \ref{fig:dynamical-bordism}. An archetypical example of a dynamical bordism of this type is $\DD \times [0,1]$ with the vertical vector field $X = \partial_t$. In this case, $X$ is inward pointing at $\DD \times \{0\}$, outward pointing at $\DD \times \{1\}$ and parallel to the rest of the boundary $\partial\DD \times (0,1)$. 
\end{remark}

\begin{figure}[H]
    \centering
    \includegraphics[scale=0.15]{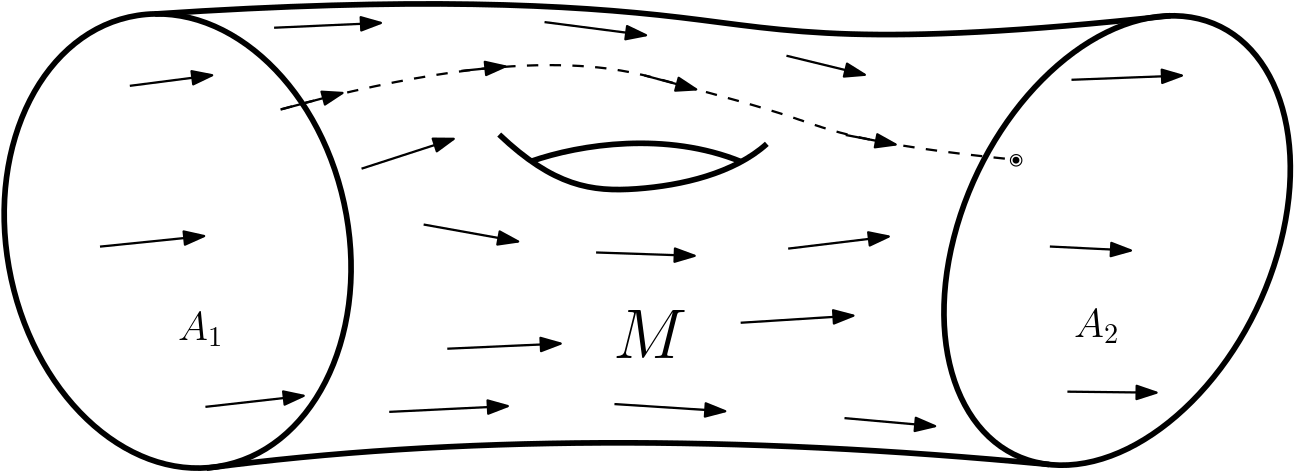}
    \caption{Dynamical bordism with the vector field on it. The dashed line represents the flow of a point.}
    \label{fig:dynamical-bordism}
\end{figure}

Now, let us focus on a dynamical bordism $M$ between standard discs $\DD_1$ and $\DD_2$. As explained in Section \ref{sec:cantor}, using the square Cantor set it is possible to define an embedding $\kappa: B^* \hookrightarrow \DD$ of the set $B^*$ of finite binary strings. But we have a natural inclusion $\NN \subseteq B^*$ by encoding every natural number as the string of its binary digits, so we get another inclusion also denoted by $\kappa: \NN \hookrightarrow \DD$. In this manner, the reaching function $\rho_{(M, X)}: \DD_1 \dashrightarrow \DD_2$ further restricts to a partial function
$$
    \rho_{(M,X)}^\kappa: \NN \dashrightarrow \NN.
$$ 

It is thus very natural to ask which partial functions $f: \NN \dashrightarrow \NN$ arise as reaching functions $\rho_{(M,X)}^\kappa$ of a dynamical bordism $(M, X)$ between discs. In this direction, the main result of \cite{GPMPS} is that partial recursive functions can be represented by dynamical bordisms.

\begin{theorem}[\cite{GPMPS}]\label{thm:hybrid}
    Given a partial recursive function $f: \NN \dashrightarrow\NN$, there exists a dynamical bordism $(M, X)$ with $M$ a bordism between two discs and $X$ volume-preserving such that the reaching function of $(M, X)$ is exactly $f$.
\end{theorem}

\begin{remark}
    Theorem \ref{thm:hybrid} is radically different from results of the style of Theorem \ref{thm:Euler-universal}. The Poincar\'e first return map in Theorem \ref{thm:Euler-universal} only represents one iteration of a Turing machine. To perform a calculation, we must follow the flow an indeterminate number of times, hitting the Poincar\'e section until the Turing machine halts and returns the output of the function. In sharp contrast, the computation of Theorem \ref{thm:hybrid} occurs in one go: when the flow reaches the output boundary, the computation has ended. This, in particular, means that $M$ is designed in such a way that the flow of $X$ remains trapped inside the bordism while it performs the desired calculation, and only escapes when the computation has been completed.
\end{remark}

\begin{remark}
    The topology of $M$ may be highly non-trivial. In fact, not every partial recursive function $f: \NN \dashrightarrow \NN$ can be extended to a continuous map $\tilde{f}: \DD \to \DD$, when $\NN \hookrightarrow \DD$ is endowed with the induced topology. In particular, the construction of Proposition \ref{prop:mapping-torus} does not work in general in this setting. This reflects the fact that $X$ may need loops and bifurcations to perform the calculation of $f$, an idea that will be further explored in Section \ref{sec:complexity-TKFT}.
\end{remark}

\begin{remark}
    The constructed dynamical bordism $(M, X)$ actually has an important property called `cleanness'. Roughly speaking, this means that $M$ is constructed by gluing tubes whose dynamics are generated by simple diffeomorphisms of the disc. This in particular implies that $X$ is generically non-zero and the dynamics of $X$ are very controlled locally, preventing pathological dynamics. For further details, see \cite{GPMPS}.
\end{remark}

The key insight used in the proof of Theorem \ref{thm:hybrid} is the fact that any Turing machine can be represented as a directed labelled graph, typically known as a \emph{finite-state machine} in the literature. Suppose that $M = (Q, q_0, Q_{\mathrm{halt}}, \delta)$ is a Turing machine with alphabet $A$. We create a graph $\mathcal{G}_M$ as follows:
\begin{itemize}
    \item The vertices of $\mathcal{G}_M$ are exactly the states $Q$ of $M$. The vertex $q_0$ is marked as the `starting vertex' and the vertices of $Q_{\mathrm{halt}}$ as `stopping vertices'.
    \item Let $q \in Q$ a state and $a \in A$ a symbol of the alphabet. Set $ (q', a', s) = \delta(q,a)$, where $s = \pm 1$. Then, we add a directed arrow in $\mathcal{G}_M$ between the vertices $q$ and $q'$ with label $(a, a', s)$. 
\end{itemize}
It is clear that the graph $\mathcal{G}_M$ contains the same information as the tuple $(Q, q_0, Q_{\mathrm{halt}}, \delta)$. But, more interestingly, this graph represents the state dynamics of $M$ in the sense that, during a calculation, the read/write head moves its state around the graph $\mathcal{G}_M$. In this way, if the current state is the vertex $q$ and the head reads $a$ in the tape, then the head transitions through the unique outgoing arrow of the form $(a, a', s)$ to the target state $q'$ and, during the transition, substitutes the symbol in the tape by $a'$ and shifts it by $s$.

\begin{proof}[Sketch of proof of Theorem \ref{thm:hybrid}]
    We shall outline the main steps involved in the proof. For a fully detailed argument, please check \cite[Appendix A]{GPMPS}. Let $f: \NN \dashrightarrow \NN$ be a partial recursive function. The dynamical bordism $(M, X)$ is constructed as follows.
    \begin{itemize}

        \item By Theorem \ref{thm:kleene}, there exists a Turing machine $M$ that computes $f$. Without loss of generality, we may assume that $M$ has a single halting state $q_f$ and that $M$ is reversible, possibly by combining halting states if necessary. Consider the graph $\mathcal{G}_M$ associated to $M$. 
        \item Thicken the graph $\mathcal{G}_M$ to convert it into a $3$-dimensional manifold with boundary. In this way, every vertex of $\mathcal{G}_M$, representing a state $q$, is converted into a $2$-dimensional disc $\DD(q)$, and every arrow in $\mathcal{G}_M$ between states $q$ and $q'$ is turned into a tube $\mathbb{D}\times [0,1]$ glued to $\mathbb{D}(q)$ and $\mathbb{D}(q')$, as shown in Figure \ref{fig:thickening}. 
\begin{figure}[H]
    \centering
    \includegraphics[scale=0.4]{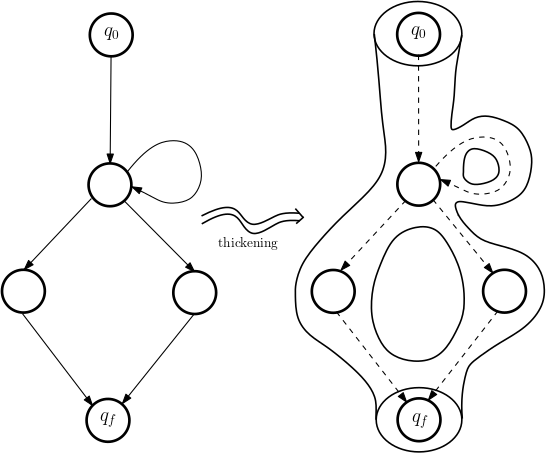}
    \caption{Thickening a finite-state machine diagram to obtain a dynamical bordism.}
    \label{fig:thickening}
\end{figure}
        \item More precisely, consider an arrow between states $q$ and $q'$ with label $(a, a', s)$, meaning that $a$ is the read symbol, $a'$ is the written symbol and $s$ is the shifting. In the tube $\DD \times [0,1]$ corresponding to the arrow, the boundary $\DD \times \{0\}$ is glued to the subset of $\DD(q)$ of those sequences starting with $a$. Analogously, the boundary $\DD \times \{1\}$ is glued to the subset of $\DD(q')$ of sequences starting with $a'$, in such a way that the gluing process substitutes the symbol $a$ by $a'$. Since $M$ is reversible, the outgoing tubes glued to the same disc are disjoint. This gives rise to a CW-complex $\widetilde{M}$ that is a bordism between the disc $\DD(q_0)$ associated to the starting state and the disc $\DD(q_f)$ corresponding to the halting state.
        \item By Theorem \ref{thm:GS-as-block}, the (generalized) shift $s$ corresponds to a diffeomorphism $f_s: \DD \to \DD$. Choose an isotopy $H_t$ between the identity morphism and $f_s$, with $H_t = \id_\DD$ for $t < \epsilon$ and $H_t = f_s$ for $t > 1-\epsilon$, and $H_t|_U = \id_U$ for all $t$ and $U$ a small open neighbourhood of $\partial \DD$. The velocity field of the flow $H_t$ defines a volume-preserving vector field on the tube $\DD \times [0,1]$ that coincides with $\partial_t$ on an open neighbourhood of the boundary $\partial(\DD \times [0,1])$ and whose reaching map between $\DD(q)$ and $\DD(q')$ is exactly $f_s$.
        \item Since the vector field on each tube is trivial around the boundaries, they can be glued to give a global vector field $\widetilde{X}$ on $\widetilde{M}$. Notice that, as $\widetilde{X}$ arises from gluing volume-preserving vector fields, it is so.
        \item Since the flow of $\widetilde{X}$ captures the read/write/shift operation of $M$, the reaching map $\rho_{(\widetilde{M},\widetilde{X})}: \DD(q_0) \dashrightarrow \DD(q_f)$ coincides exactly with the initial partial recursive function $f$ when restricted to $\NN \hookrightarrow \DD(q_0)$ and $\NN \hookrightarrow \DD(q_f)$.
        \item Smooth out the CW-complex $\widetilde{M}$ into a smooth manifold $M$ that contains $\widetilde{M}$ as a strong deformation retract. The vector field $\widetilde{X}$ can also be extended to a vector field $X$ on $M$ without altering its value on the skeleton $\widetilde{M}$. In particular, the reaching function $\rho_{(M, X)}: \DD(q_0) \dashrightarrow \DD(q_f)$ restricts to $f$, as we wanted to prove.
    \end{itemize}
\end{proof}

\begin{remark}
    The dynamical bordism constructed in Theorem \ref{thm:hybrid} is equipped with a volume-preserving vector field $X$. In an ongoing research with K.\ Cieliebak, we will show that a modification of this construction yields a vector field $X$ on any closed $3$-manifold (which is no longer a dynamical bordism) that solves the stationary Euler equations for some Riemannian metric. This will show that it is possible to embed computability in Euler flows in a sense complementary to Theorem \ref{thm:Euler-universal}. 
\end{remark}

We finish this section by pointing out that Theorem \ref{thm:hybrid} has an alternative interpretation in categorical terms. Recall that a category is a collection of objects and a collection of morphisms between them that can be composed if their domain and codomain coincide. In particular, the partial functions form a category $\textbf{PF}$ with objects the sets of the form $\NN^n$ for some $n \geq 0$, and a morphism is a partial function $f: \NN^n \dashrightarrow \NN^m$. Inside the category $\textbf{PF}$, we can find the subcategory $\textbf{PRF}$ with the same objects but whose morphisms are only partial recursive functions.

Analogously, we can consider the category $\textbf{Bord}_2^{\textup{dy}}$ whose objects are disjoint unions of discs and whose morphisms are dynamical bordisms between them. In this sense, the reaching function construction leads to a functor
\begin{equation}\label{eq:pseudo-TKFT}
    \rho: \textbf{Bord}_2^{\textup{dy}} \longrightarrow \textbf{PF}.
\end{equation}
In this context, in \cite{GPMPS}, we show that if we restrict to the subcategory $\textbf{Bord}_2^{\textup{cln}}$ of $\textbf{Bord}_2^{\textup{dy}}$ of clean bordisms, the result obtained is always a partial recursive function. Hence, the functor (\ref{eq:pseudo-TKFT}) can be restricted to a functor
\begin{equation}\label{eq:TKFT}
    \rho|_{\textbf{Bord}_2^{\textup{cln}}}: \textbf{Bord}_2^{\textup{cln}} \longrightarrow \textbf{PRF}.
\end{equation}
Furthermore, Theorem \ref{thm:hybrid} shows that this functor is full, i.e., every morphism in the target category $\textbf{PRF}$ is in the image of a morphism in $\textbf{Bord}_2^{\textup{cln}}$. In this abstract sense, the functor (\ref{eq:TKFT}) synthesizes the strong relation between dynamics and computability. For this reason, we call (\ref{eq:TKFT}) a \emph{Topological Kleene Field Theory (TKFT)}, named after S.\ Kleene who introduced partial recursive functions as a computational model. For further details, please see \cite{GPMPS}.

It is worth noticing that the name TKFT also highlights the resemblance between (\ref{eq:TKFT}) and a Topological Quantum Field Theory (TQFT), understood as a quantum field theory independent of the underlying metric. In this setting, as formalized by Atiyah in \cite{atiyah1988topological}, a TQFT forms a functor
$$
    Z: \textbf{Bord}_n \longrightarrow \textbf{Vect}
$$
from the category of $n$-dimensional bordisms to the category of vector spaces. In this language, in the same way that a TQFT captures the essence of quantization by means of Hilbert spaces, a TKFT encapsulates the idea that a dynamical bordism represents a computable function.

\subsection{Complexity through dynamical bordisms}\label{sec:complexity-TKFT}

An important consequence of Theorem \ref{thm:hybrid} is that it enables the study of computable function complexity from a purely topological perspective via the associated dynamical bordism.

\begin{definition}
    The \emph{topological complexity} of a partial recursive function $f$ is
    $$
        \textbf{TopC}(f) = \inf_{M} b_1(M),
    $$
    where the infimum runs over all clean dynamical bordisms with reaching function $f$ and $b_1(M)$ denotes the first Betti number of $M$, that is, the rank of the first homology group of $M$.
\end{definition}

\begin{remark}
    Intuitively, $b_1(M)$ counts the number of loops in $M$, and therefore $\textbf{TopC}(f)$ is the minimum number of loops of that a manifold must have to compute $f$. Notice that, since $M$ is connected and retracts to a $1$-dimensional CW-complex, it follows that the $0$-th Betti number is $b_0(M)=1$ and $b_n(M) = 0$ for $n\geq 2$. Therefore, $b_1(M)$ is the only non-trivial invariant of this type.
\end{remark}

This invariant is closely related to the so-called loop complexity $\textbf{LoopC}(f)$ of $f$. { This is defined as the minimum number of loops that a flow chart for an algorithm computing $f$ must have. Notice that the proof of Theorem \ref{thm:hybrid} actually provides a bound for the topological complexity in terms of the loop complexity.}

{
\begin{proposition}\label{prop:complexity}
    For any computable function $f$, we have
$$
    \textbf{\textup{TopC}}(f) \leq \textbf{\textup{LoopC}}(f).
$$
\end{proposition}

\begin{proof}
    Consider a flowchart $F$ for an algorithm with $F_n$ loops. This flowchart directly translates into a Turing machine graph $\mathcal{G}$ with the same number $n$ of loops in the graph. Therefore, by the thickening construction of Theorem \ref{thm:hybrid}, the bordism $M$ associated to this Turing machine has $\mathcal{G}$ as deformation retract, so $b^1(M) = b^1(\mathcal{G}) = F_n$. This implies that
    $$
        \textbf{\textup{TopC}}(f) \leq F_n,
    $$
    for any flowchart $F$. Taking the infimum on the right hand side leads the desired inequality. 
\end{proof}
}

{Notice that there may exist alternative dynamical bordisms realizing the computable function $f$ that are not derived from the thickening of a Turing machine, which are the ones used in the proof of the bound of Proposition \ref{prop:complexity}. This implies that, if the calculation could be done more efficiently by a dynamical bordism not corresponding to a Turing machine, the inequality would be strict. This would suggest the existence of a fundamental gap, implying that dynamical bordisms could outperform Turing machines in computational efficiency.
}

However, in complexity theory, the important quantity is the number of operations performed by an algorithm to compute the output $f(n)$ as a function of the input $n$. In this dynamical context, there exists an analogous geometric feature that measures the length of the path needed to compute $f(n)$ through a dynamical bordism.

\begin{definition}
    Let $(M, X)$ be a dynamical bordism with reaching function $\rho_{(M, X)}: \NN \dashrightarrow \NN$, and consider a Riemannian metric $g$ on $M$ ensuring that $X$ is divergence-free. The \emph{length complexity} is the function $\textbf{LenC}_g(f): \NN \to (0, \infty]$ given by
    $$
        \textbf{LenC}_{(M,g)}(n) = \int_{0}^{t_n} ||X_{\varphi_t(n)}||_g \,dt,
    $$
    where $\varphi_t$ denotes the flow of $X$ initiated at the incoming boundary point $n \in \NN\hookrightarrow M$, and $t_n$ is the time at which $\varphi_{t_n}(n)$ reaches the outgoing boundary of $M$.
\end{definition}

\begin{remark}
    The above definition makes sense even if $\varphi_t(n)$ never reaches the outgoing boundary of $M$ since in that case $t_n = \infty$.
\end{remark}

\begin{remark}
    In the case that there exists a Riemannian metric $g$ such that $X$ is a Killing vector field for $g$, then $\varphi_t$ is an isometry for each $t$ and thus the norm $||X_{\varphi_t(n)}||_g$ is constant in $t$. If we further require the $||\partial_t||_g = 1$ around the incoming boundary, then we have that $\textbf{LenC}_{(M,g)}(n) = t_n$, so the length complexity is exactly the time needed to flow until the outgoing boundary.
\end{remark}

Notice that the length complexity $\textbf{LenC}_{(M,g)}$ inherently depends on the choice of metric $g$. In particular, if we rescale $g' = \lambda g$ for some $\lambda > 0$, then $\textbf{LenC}_{(M,g')} = \lambda \textbf{LenC}_{(M,g)}$. However, this rescaling does not alter the asymptotic order of magnitude of $\textbf{LenC}_{(M,g)}$, meaning that the ratio $\textbf{LenC}_{(M, g')}(n)/\textbf{LenC}_{(M,g)}(n)$ converges as $n\to\infty$, as it is typically important in complexity theory. This is particularly interesting if we want to compute a ``weighted complexity'', in which some inputs count less than others. This makes sense, for instance, if we want to compute some kind of average complexity with respect to a non-uniform probability distribution in the space of inputs.

Finally, observe that since a clean bordism $M$ is assembled by gluing standard tubes along discs, it naturally carries a distinguished metric $g_0$, which is flat within each tube. For this choice, $\textbf{LenC}_{(M,g_0)}$ computes the usual Euclidean length of the flow. {We conjecture that the length with respect to this metric is actually asymptotically comparable to the usual time complexity of the function it computes.

\begin{conjecture}
    Let $f: \NN \dashrightarrow \NN$ be a partial recursive function. Given a Turing machine $\mathcal{G}$, for any $n \in \NN$ let $T(n)$ be the number of steps that $M$ performs when started with input $n$ before halting. Then, if $M_\mathcal{G}$ is the dynamical bordism resulting from thickening the graph $\mathcal{G}$, then 
    $$
    \lim_{n \to \infty} \frac{\textbf{\textup{LenC}}_{(M_\mathcal{G},g_0)}}{T(n)} = 1.
    $$
\end{conjecture}


Notice that, if this conjecture holds true, this result would provide us with new ways to approach the study of computational complexity in purely geometric terms.
}


\section{Open Problems and Future Directions}

The study of universality in dynamical systems and its connections with computation raises numerous challenging questions.
The synthesis of dynamical systems and computability via TKFT (Topological Kleene Field Theory) offers a new framework for understanding how continuous dynamics can emulate computation. However, many questions remain open. In this section, we highlight some refined open problems that should foster further research in the area.

\begin{enumerate}[label=\arabic*.]
    \item \textbf{Topology and Complexity in TKFT:} Explore the precise relationship between the computational complexity of partial recursive functions and the topology of their representing bordisms. {Specifically,} can topological invariants be formulated to quantify computational complexity?
    Can alternative classes of flows exhibit Turing-complete behaviors, potentially enhancing computational efficiency? 
\end{enumerate}

{From a broader perspective, these TKFTs open a new connection between continuous dynamical systems and computability. Thanks to this bridge, it is expected that the complexity properties of computable functions translate into intrinsic topological features of the representing bordisms. In this direction, we suggest that future work should further explore analogies with quantum computing, complexity theory, or fluid dynamical systems.
}

\begin{enumerate}[label=\arabic*.]
  \setcounter{enumi}{1}
    \item \textbf{Identify the dynamical essence of Turing-universality:} It is well known that chaotic systems are characterized by the presence of a horseshoe map, in the sense that 3D conservative dynamics exhibits chaos if and only if it contains a horseshoe. In this direction, is it possible to identify a similarly fundamental operation that characterizes Turing universality such that a family of dynamical systems is Turing universal if and only if it contains this operation?
\end{enumerate}

{
Addressing this problem would shed light to a more fundamental question: What makes a dynamical system Turing-universal, in the sense of Definition \ref{defn:Turing-universal}? It would be very interesting to understand whether Turing universality is a purely topological phenomenon that can be isolated in a simple operation, as is the horseshoe chaos, or, on the contrary, it requires the global interplay of different behaviors.
}

\begin{enumerate}[label=\arabic*.]
  \setcounter{enumi}{2}
    \item \textbf{ Transverse vector fields and geometrical structures:} Analyze how additional constraints, such as requiring vector fields to be Beltrami or satisfy specific energy conditions, affect universality properties. In particular, analyze the case of modular vector fields in Poisson dynamics as generalizations of Reeb vector fields.
\end{enumerate}

{
Imposing extra constraints to vector fields confers them with additional rigidity that makes the realization problem harder but, in contrast, the resulting vector fields may have a physical interpretation as motion flows of a mechanical system or solutions to physical problems. For instance, the Poisson structure is the one inherited by the algebra of smooth functions in a symplectic manifold, thus capturing the algebra of observables of a mechanical system. Notice that this kind of rigidity might be expected in many cases, as shown by Theorem \ref{thm:Euler-universal}, where the vector field can be taken to be Reeb-like for a certain contact structure.
}

\begin{enumerate}[label=\arabic*.]
  \setcounter{enumi}{3}
  
    \item \textbf{Comparison with super-Turing models and quantum supremacy:} We propose exploring the ultimate computational limits of continuous dynamical systems in terms of both capability and efficiency. {Due to the presence of new topological techniques for dynamical system, it could be possible to exploit them to speed up the computation of certain functions}. In particular, can a dynamical system be engineered that surpasses the computational power of a quantum computer, beating the quantum supremacy?

\end{enumerate}

{
In \cite{cardona2024hydrodynamic}, it was shown that area-preserving diffeomorphisms of the disc can encode the dynamics of a Turing machine with advice. This enables the embedding of super-Turing systems as Euler flows or dynamical bordisms, allowing the representation of non-computable functions and achieving the complexity class \textbf{P/poly}. This observation suggests that the TKFT constructions presented in this work may have the potential to supersede classical computation. Particularly, ``beyond-shift'' diffeomorphisms of the disc may provide a way of performing hard computations in a short time, in the same way quantum entanglement allows quantum computing to speed up certain auxiliary calculations.
}

\begin{enumerate}[label=\arabic*.]
  \setcounter{enumi}{4}
  
    \item \textbf{Physical realization of computational dynamics:} To what extent can TKFT-inspired dynamical systems be implemented in physical systems, such as fluid dynamics or electromagnetic fields? {In particular, is it possible to construct a physical device able to realize the Turing-universal vector fields with arbitrary precision?}
\end{enumerate}

{
Dynamical systems representing Turing universality rely on intricate properties, including adapted metrics for Turing-complete fluids and complex geometries in TKFTs. While these abstract systems are not in principle physically realizable, identifying a subclass of computational problems whose dynamical systems can be physically implemented remains an intriguing open question. This would provide a roadmap to a prospective construction of a physical device able to perform these calculations by exploiting the potential of physical systems to compute.

It is worth mentioning that the problem of the physical realizability of Turing-universal systems has been a central concern since the very beginning of computation, with special attention to the celebrated papers by Feynman \cite{feynman2018simulating} and Wolfram \cite{wolfram1985undecidability}. In addition, in the recent work \cite{jaeger2023toward}, the authors address the problem of physical realizability of computers by shift from classical Turing-based models to a bottom-up, physics-grounded theory that models computation as the structuring of measurable physical processes rather than symbolic manipulation.
}

\begin{enumerate}[label=\arabic*.]
  \setcounter{enumi}{5}
    \item {\textbf{Realization of computational dynamics as chaotic systems in celestial mechanics:} Is the $n$-body problem in celestial mechanics Turing-universal for some $n > 2$  (and for certain choices of masses)? Is the $n$-centre problem in celestial mechanics Turing-universal for some $n$? Is the generalized $n$-centre problem in celestial mechanics Turing-universal for some $n > 2$ and for some metric?}
\end{enumerate}

A natural class of physical systems that interfaces well with our framework consists of those governed by Reeb-type geometric structures, such as contact and cosymplectic geometries. In these contexts, the system dynamics typically emerges as the restriction of a Hamiltonian flow to a fixed energy level set within a symplectic manifold (the phase space). This perspective has been fruitfully applied to the study of the restricted three-body problem, as developed in \cite{AFKP}, where contact topology plays a crucial role in proving the existence of periodic orbits.

We now seek an analogous paradigm: one that bridges classical dynamical complexity, such as chaos, with notions of Turing universality or logical undecidability. Although the chaotic nature of various systems in celestial mechanics has been extensively studied (see, e.g., \cite{baldomaetal}), to the best of our knowledge, the computational power of the $n$-body problem has not yet been analyzed through the lens of Turing machines.

On the other hand, the study of the $n$-centre problem has revealed intriguing connections with symbolic dynamics and the Cantor set; see \cite{knauf}. A closely related system, the generalized $n$-centre problem on Riemannian surfaces, also exhibits rich chaotic behaviour, as demonstrated in \cite{new}. This strongly suggests that these problems should exhibit some type of Turing universality, analogously to that encountered in fluid dynamics.

{
\section{Conclusions}

This work has surveyed a range of recent developments illustrating the deep interplay between dynamical systems and computation, particularly through the lens of universality. Beginning with classical embeddings of dynamical systems into higher dimensional ones, we have traced how increasingly complex dynamical behavior can simulate intricate processes.

Using the notion of Poincar\'e first return map, we introduced a definition of Turing universality of a continuous dynamical system. This culminates in the remarkable universality of Euler flows, capable of reproducing the dynamics of Turing machines through the Poincaré return maps of stationary solutions. The results reviewed here show that such flows can simulate any algorithmic behavior, establishing fluid motion not merely as a physical process but as a computational medium.

Additionally, we have further explored a new paradigm: the representation of computable functions via dynamical bordisms. These constructions, which we term Topological Kleene Field Theories (TKFTs), enrich the traditional correspondence between computation and geometry by encoding computation as reaching maps of vector fields on manifolds. This viewpoint leads naturally to novel notions of complexity, both topological and metric, which open the door to a new, geometrically grounded complexity theory.

The connections laid out here suggest numerous avenues for future work, ranging from identifying topological invariants that govern computability, to exploring whether such dynamics can be physically realized or even leveraged to rival quantum computation. At the intersection of dynamics, topology, and theoretical computer science, the theory of dynamical universality offers both a unifying perspective and fertile ground for new insights.
}

\bibliographystyle{abbrv}
\bibliography{bibliography}

\end{document}